\documentclass[10pt]{article}%
\usepackage{amsmath}
\usepackage{amsfonts}
\usepackage{amssymb}
\usepackage{graphicx}%
\usepackage{mathrsfs}
\usepackage[square, comma, sort&compress, numbers]{natbib}
\usepackage{amssymb, color}
\usepackage[body={15.5cm,21cm}, top=3cm]{geometry}%
\DeclareMathOperator*{\esssup}{ess\,sup}
\providecommand{\U}[1]{\protect \rule{.1in}{.1in}}

\newtheorem{theorem}{Theorem}[section]

\newtheorem{definition}[theorem]{{Definition}}
\newtheorem{example}[theorem]{{ Example}}

\newtheorem{lemma}[theorem]{Lemma}

\newtheorem{remark}[theorem]{{Remark}}

\newenvironment{proof}[1][Proof]{\noindent \textbf{#1.} }{\  \rule{0.5em}{0.5em}}
\begin{document}

\title{Maximum principle for  stochastic recursive optimal control problem under model uncertainty}
\author{Mingshang Hu \thanks{Zhongtai Securities Institute for Financial Studies, Shandong University. humingshang@sdu.edu.cn. Hu and Wang's research was
supported by the National Key R\&D Program of China (No. 2018YFA0703900). Hu's research was supported by the National Natural Science Foundation of China (No. 11671231) and the Young Scholars Program of Shandong University (No. 2016WLJH10).}
\and Falei Wang\thanks{Zhongtai Securities Institute for Financial  Studies, Shandong University.
flwang2011@gmail.com (Corresponding author). Wang's  research was supported by   the National Natural Science Foundation of China (No. 11601282), the Natural Science Foundation of Shandong Province (No. ZR2016AQ10)  and the Young Scholars Program of Shandong University.}}
\date{}
\maketitle
\begin{abstract}
In this paper, we consider a stochastic recursive optimal control problem under model uncertainty. In this framework, the cost function is described by solutions of a family of backward stochastic differential equations
with uncertainty parameter $\theta$, which is used to represent  different market conditions. With the help of linearization techniques and weak convergence methods, we derive the corresponding stochastic maximum principle. Moreover, a linear quadratic robust control problem is also studied.
\end{abstract}


\textbf{MSC-classification}: 93E20, 60H10, 35K15

\section{Introduction}
The nonlinear backward stochastic differential equations (BSDEs)  formulated by Pardoux and Peng \cite{PP1}, provided a powerful tool for the research of stochastic control problem and partial differential equations. In particular, El Karoui, Peng, and Quenez \cite{EP} applied BSDEs to characterize the so-called stochastic recursive optimal control problem. In this framework,      the asset price is  described by $x$ term and the cost function is defined by $y(0)$ term of the following  forward and backward stochastic differential equation (FBSDE) on a  finite time horizon $[0,T]$:
\begin{align} \label{myq9001}
\begin{cases}
&x(t)=x_0+\int^t_0b(s,x(s),u(s))ds+\int^t_0\sigma(s,x(s),u(s))dW(s),\\
&y(t)=\varphi(x(T))+\int^T_tf(s,x(s),y(s),z(s),u(s))ds-\int_{t}^{T}z(s)dW(s),
\end{cases}
\end{align}
where $W=(W(t))_{0\leq t\leq T}$ is a  standard $d$-dimensional Brownian motion on a complete probability space $(\Omega,\mathscr{F},\mathbb{P})$ and
$u$ denotes an admissible control process taking value in some nonempty set $U$ (see Section 2 for more details).

The stochastic recursive optimal control problems have
 important applications in mathematical finance and engineering. For instance, Chen and
Epstein \cite{CE} considered the stochastic differential recursive utility with drift
ambiguity, which can be
characterized  by a special kind of BSDE (see also Duffie and Epstein \cite{DE}).
Moreover, the equation \eqref{myq9001} reduces to the classical  stochastic optimal control problem when the generator $f$ is independent of the arguments $y$ and $z$.

In practice, taking into account the  model uncertainty, it is hard to know the actual  drift and diffusion  coefficients of $x$.
For example,  the share market is often described as being in either a bull market or a bear market. However,
the  coefficients may be different in a bull market and a bear market. Since bull markets or bear markets are difficult to predict, we do not know if the actual cost is $y_1(0)$ or $y_2(0)$, where $y_{1}(0)$ and $y_2(0)$ represent the costs in a bull market and a bear market, respectively.
Suppose that the probability $\lambda$ of a bull market occurring is unknown.
In this case, we could measure  the cost in the following robust way
\begin{align}\label{myq90001}
J(u)=\sup\limits_{\lambda\in[0,1]}\left(\lambda y_1(0)+(1-\lambda)y_2(0)\right)=\max(y_1(0),y_2(0)),
\end{align}
which can be regarded  as a  robust optimal control problem.

In the sequel, we  use the parameter $\theta\in\Theta$ to  represent   different market conditions, where $\Theta$ is a locally compact Polish space. The corresponding cost $y_{\theta}(0)$ is given by
\begin{align} \label{myq90011}
\begin{cases}
&x_{\theta}(t)=x_0+\int^t_0b_{\theta}(s,x_{\theta}(s),u(s))ds+\int^t_0\sigma_{\theta}(s,x_{\theta}(s),u(s))dW(s),\\
&y_{\theta}(t)=\varphi_{\theta}(x_{\theta}(T))+\int^T_tf_{\theta}(s,x_{\theta}(s),y_{\theta}(s),z_{\theta}(s),u(s))ds-\int_{t}^{T}z_{\theta}(s)dW(s),
\end{cases}
\end{align}
where the coefficients  of  the controlled FBSDEs depends on the market uncertainty parameter $\theta$.
Suppose that $\mathcal{Q}$ is the  set of  all possible probability distributions of  $\theta$.
Then, the  robust cost function is defined by
\begin{align}\label{myq90002}
J(u)=\sup\limits_{Q\in\mathcal{Q}}\int_{\Theta} y_{\theta}(0)Q(d\theta).
\end{align}
It is obvious that equation \eqref{myq90001} is a special case of equation \eqref{myq90002}.
Thus, an interesting question is to study the above stochastic recursive  optimal  robust control problem.

An important approach for optimal control problems is to derive  maximum principle, namely, necessary condition for optimality.
In the seminal paper \cite{P1}, Peng established a global maximum principle for the classical
stochastic optimal control problem.
Since then, the stochastic maximum principle was extensively investigated
for various stochastic systems, such as mean field dynamics, infinite-dimensional case and so on. Indeed, Buckdahn,  Li and  Ma  \cite{BL1} studied the optimal control problem for
mean-field SDEs; Fuhrman, Hu and  Tessitore \cite{FHT} considered maximum
principle for infinite-dimensional stochastic control systems;
Tang \cite{Ts} obtained a general partially observed maximum principle with correlated noises between the system and the observation.
For more research on this topic, the reader is referred to \cite{DM,FHT1,HP1,HP,HT,LZ,QT, TL,Yu}
and the references therein.

Furthermore, much research is also devoted to  studying maximum principle for the stochastic recursive optimal control problems.
Peng \cite{P2} first studied the  convex control domain case and established a  local maximum principle.
Then, Ji and Zhou \cite{JZ} obtained a local maximum principle for the convex case with terminal state constraints.
 Xu \cite{XW} considered the nonconvex case when the diffusion coefficient does not include control variable.
We refer the reader to \cite{H,HJ1,HJ,P3,Y} for a closest related research.

The present paper is devoted to the research of stochastic maximum principle for the above stochastic recursive optimal robust control problem.
In order to illustrate the main idea, we will study the convex control domain case.
Note that the robust cost is a supremum  over a family of  probability measures. Thus, the classical convex variational approach cannot be directly applied to this question. To overcome this difficult, we deal with the derivative of the value function through  weak convergence methods.

In order to carry out the purpose, we assume that $\mathcal{Q}$ is weakly compact and convex. With the help of the linearization techniques,
we obtain the variation equation of the FBSDE for each  uncertainty parameter $\theta$. Unlike the classical case, we need to establish the convergence for the variational equation uniformly with respect to $\theta$.
Then, in the spirit of  Sion's minimax theorem,  we prove that the variational
inequality is the integral of  the variational BSDE  with respect to a reference probability $\overline{Q}\in\mathcal{Q}$.
We also study the  regularity of the Hamiltonian function  to deal with some measurability issues with respect to  the parameter $\theta$.
 Based on the above results, the stochastic maximum principle is derived. Moreover,  the stochastic maximum principle is also a sufficient
condition under some convex assumptions.

The contribution of this paper is threefold. First, the stochastic recursive optimal robust control problem
under model uncertainty is formulated. In particular,  the robust cost involves a family of cost functions under different market conditions.
Next, the stochastic maximum principle  is obtained, which involves the integral of  the Hamiltonian function with respect to the above probability $\overline{Q}\in\mathcal{Q}$.  To the best of our knowledge, this is the first study to the above type of maximum principle.
Finally, we apply the maximum principle to solving a linear quadratic robust control problem. Moreover, compared with \cite{P2}, our problem is essentially an ``inf sup problem'', which makes it more delicate and challenging.

The paper is organized as follows. In section 2, we formulate the  stochastic recursive optimal  robust control problem. Then, we state the maximum principle in section 3. The section 4 is devoted to the study of a linear quadratic robust control problem.

\subsubsection*{Notation.}
Throughout this paper, let $(\mathscr{F}_t)_{0\leq t\leq T}$ be the  natural filtration generated by $W$ augmented by the $\mathbb{P}$-null sets of $\mathscr{F}$.
For  each Euclidian space, we  denote by $\langle\cdot,\cdot \rangle$  and  $|\cdot|$
 its scalar product and the associated norm, respectively.
Denote by $\mathbb{R}^n$ the $n$-dimensional real
Euclidean space, $\mathbb{R}^{n\times d}$ the set of $n\times d$ real matrices and $\mathbb{S}_n$ the set of symmetric $n\times n$ real matrices. Moreover, we use the notation $\partial_x=(\frac{\partial}{\partial x^1},\cdots,\frac{\partial}{\partial x^n})$, for $x\in\mathbb{R}^n$. Then, $\partial_x\psi=(\frac{\partial\psi}{\partial x^1},\cdots,\frac{\partial\psi}{\partial x^n})$ is a row vector  for $\psi:\mathbb{R}^n\rightarrow \mathbb{R}$ and  $\partial_x\Psi=[\frac{\partial\Psi^i}{\partial x^j}]$ is a $d\times n$ matrix  for $\Psi=(\Psi^1,\cdots,\Psi^d)^{\top}:\mathbb{R}^n\rightarrow \mathbb{R}^d$.
Finally, we consider the following Banach spaces: for any $p\geq 1$,
\begin{description}
\item[$\bullet$] ${L}^{p}(\mathscr{F}_T;\mathbb{R}^n)$ is the space of  $\mathbb{R}^n$-valued  $\mathscr{F}_T$-measurable random vectors $\xi$
satisfying
$
\mathbb{E}[|\xi|^p]<\infty;
$
\item[$\bullet$]  $\mathcal{M}^p(0,T;\mathbb{R}^n)$  is the space of   $\mathbb{R}^n$-valued  $\mathscr{F}$-progressively measurable   processes $(u(t))_{0\leq t\leq T}$
satisfying
\begin{align*}
\mathbb{E}\left[\int^T_0|u(t)|^pd\right]<\infty;\end{align*}
\item[$\bullet$]  $\mathcal{M}^{\infty}(0,T;\mathbb{R}^n)$  is the space of   $\mathbb{R}^n$-valued  $\mathscr{F}$-progressively measurable   processes $(u(t))_{0\leq t\leq T}$
satisfying
\begin{align*}
\esssup\limits_{(t,\omega)}|u(t)|<\infty;\end{align*}
\item[$\bullet$]  $\mathcal{H}^p(0,T;\mathbb{R}^n)$  is the space of   $\mathbb{R}^n$-valued  $\mathscr{F}$-progressively measurable   processes $(z(t))_{0\leq t\leq T}$
satisfying
\begin{align*}
\mathbb{E}\left[\left(\int^T_0|z(t)|^2dt\right)^{\frac{p}{2}}\right]<\infty;\end{align*}
\item[$\bullet$]  $\mathcal{H}^{1,p}(0,T;\mathbb{R}^n)$  is the space of   $\mathbb{R}^n$-valued  $\mathscr{F}$-progressively measurable  processes $(z(t))_{0\leq t\leq T}$
satisfying
\begin{align*}
\mathbb{E}\left[\left(\int^T_0|z(t)|dt\right)^p\right]<\infty;\end{align*}
\item[$\bullet$]  $\mathcal{S}^{p}(0,T;\mathbb{R}^n)$  is the space of $\mathbb{R}^n$-valued  $\mathscr{F}$-adapted continuous  processes $(y(t))_{0\leq t\leq T}$ satisfying
\begin{align*}
\mathbb{E}\left[\sup_{0\leq t\leq T}|y(t)|^p\right]<\infty;
\end{align*}
\item[$\bullet$]  $C(0,T;\mathbb{R}^n)$  is the space of $\mathbb{R}^n$-valued  continuous functions on $[0,T]$.
 \end{description}
In the sequel, for a given set of parameters $\alpha$, $C(\alpha)$ will denote a positive constant only depending on these parameters,
and which may change from line to line.

\section{Formulation of the problem}

We now introduce the definition of admissible control.  Assume  $U$ is a given nonempty convex subset of
$\mathbb{R}^{k}$ and $p>4$.
\begin{definition}
$u:[0,T]\times\Omega\rightarrow U$ is said to be an admissible control, if $u\in\mathcal{M}^p(0,T;\mathbb{R}^k)$.
The set of admissible controls is denoted by $\mathcal{U}[0,T]$.
\end{definition}

In the market, assume that the agent can
choose an admissible control $u\in\mathcal{U}[0,T]$ to obtain some SDE on $[0,T]$.
However, he does not know the actual drift and diffusion coefficients due to the model uncertainty. Instead,
the agent just knows a family of coefficients which may occur in the market.

In this case,   the corresponding SDE can be described by
\begin{align} \label{App1y}
x_{\theta}(t)=x_0+\int^t_0b_{\theta}(s,x_{\theta}(s),u(s))ds+\int^t_0\sigma_{\theta}(s,x_{\theta}(s),u(s))dW(s),
\end{align}
where $\theta\in\Theta$  and   $\Theta$ is a  locally compact, complete separable space with  distance $\mu$.
The corresponding cost is given by $y_{\theta}(0)$ term of the following BSDE on $[0,T]$:
\begin{align} \label{App1y11}
y_{\theta}(t)=\varphi_{\theta}(x_{\theta}(T))+\int^T_tf_{\theta}(s,x_{\theta}(s),y_{\theta}(s),z_{\theta}(s),u(s))ds-\int_{t}^{T}z_{\theta}(s)dW(s).
\end{align}
In the above equations, $b_{\theta}:[0,T]\times\mathbb{R}^{n}\times U\rightarrow
\mathbb{R}^{n}$, $\sigma_{\theta}=[\sigma_{\theta}^1,\cdots,\sigma_{\theta}^d]:[0,T]\times\mathbb{R}^{n}\times U\rightarrow
\mathbb{R}^{n\times d}$, $\varphi_{\theta}:\mathbb{R}^{n}\rightarrow
\mathbb{R}$, $f_{\theta}:[0,T]\times\mathbb{R}^{n}\times\mathbb{R}\times\mathbb{R}^{d}\times U\rightarrow
\mathbb{R}$ are Borel measurable functions.
Note that the process $(x_{\theta},y_{\theta},z_{\theta})$ depends on $u$ and we omit the superscript $u$ for convenience,
unless otherwise specified.

Due to the model uncertainty, the cost function is defined by:
\[
J(u)=\sup\limits_{Q\in\mathcal{Q}}\int_{\Theta} y_{\theta}(0)Q(d\theta),
\]
where  $\mathcal{Q}$ is a set of probability measures on $(\Theta,\mathcal{B}(\Theta))$.
Note that at this stage, we cannot even conclude that the function $ \theta\rightarrow y_{\theta}(0)$ is measurable.

In this paper, we make use of the following assumptions.
\begin{description}
\item[(H1)] There exists some positive constant $L$ such that for any $t\in[0,T], x,x^{\prime}\in\mathbb{R}^{n}, y, y^{\prime}\in\mathbb{R},z,z^{\prime}\in\mathbb{R}^{d}, u,u^{\prime}\in U,\theta\in\Theta$,
\begin{align*}
&|b_{\theta}(t,x,u)-b_{\theta}(t,x^{\prime},u^{\prime})|+|\sigma_{\theta}(t,x,u)-\sigma_{\theta}(t,x^{\prime},u^{\prime})|\leq
L(|x-x^{\prime}|+|u-u^{\prime}|),\\
&  |\varphi_{\theta}(x)-\varphi_{\theta}(x^{\prime})|+|f_{\theta}(t,x,y,z,u)-f_{\theta}(t,x^{\prime},y^{\prime},z^{\prime},u^{\prime})|\\
&  \leq L\left((1+|x|+|x^{\prime}|+|u|+|u^{\prime}|)(|x-x^{\prime}|+|u-u^{\prime}|)+|y-y^{\prime
}|+|z-z^{\prime}|\right),\\
&|b_{\theta}(t,0,0)|+|\sigma_{\theta}(t,0,0)|+|f_{\theta}(t,0,0,0,0)|\leq L.
\end{align*}
\item[(H2)] $b_{\theta},\sigma_{\theta},\varphi_{\theta},f_{\theta}$ are  continuously differentiable in $(x,y,z,u)$ for any $\theta\in\Theta$.
\item[(H3)] There exists a modulus of continuity $\overline{\omega}:[0,\infty)\rightarrow[0,\infty)$  such that
   \begin{align*}
   |\ell_{\theta}(t,x,y,z,u)-\ell_{\theta}(t,x^{\prime},y^{\prime},z^{\prime},u^{\prime})|\leq \overline{\omega}(|x-x^{\prime}|+|y-y^{\prime
}|+|z-z^{\prime}|+|u-u^{\prime}|),
    \end{align*}
for any $t\in[0,T], x,x^{\prime}\in\mathbb{R}^n, y, y^{\prime}\in\mathbb{R}, z, z^{\prime}\in\mathbb{R}^d, u,u^{\prime}\in U$, $\theta\in\Theta$,    where $\ell_{\theta}$ is the derivative of $b_{\theta},\sigma_{\theta},\varphi_{\theta}, f_{\theta}$ in $ (x,y,z,u)$.
\item[(H4)]  For each $N>0$, there exists a modulus of continuity $\overline{\omega}_N:[0,\infty)\rightarrow[0,\infty)$  such that
   \begin{align*}
   |\ell_{\theta}(t,x,y,z,u)-\ell_{\theta^{\prime}}(t,x,y,z,u)|\leq \overline{\omega}_N(\mu(\theta,\theta^{\prime})),
    \end{align*}
for any $t\in[0,T], |x|, |y|, |z|, |u|\leq N$, $\theta,\theta^{\prime}\in\Theta$,  where $\ell_{\theta}$ is $b_{\theta},\sigma_{\theta},\varphi_{\theta}, f_{\theta}$ and their derivatives in $(x,y,z,u)$.

\item[(H5)] $\mathcal{Q}$ is a
weakly compact and convex set of probability measures on $(\Theta,\mathcal{B}(\Theta))$.
\end{description}

\begin{example}{\upshape
		Let $\Theta$ be a countable discrete space. Then, $\mu(\theta,\theta^{\prime})=\mathbf{1}_{\theta\neq\theta^{\prime}}$. Thus, under assumptions
		{(H1)}-{(H3)}, it is easy to check that condition (H4) holds.
	}
\end{example}

\begin{lemma}\label{myw301} Assume that \emph{(H1)} holds. Then, the FBSDE \eqref{App1y} and \eqref{App1y11} admits a unique solution $(x_{\theta},y_{\theta},z_{\theta})\in
\mathcal{S}^p(0,T;\mathbb{R}^n)\times \mathcal{S}^{\frac{p}{2}}(0,T;\mathbb{R})\times  \mathcal{H}^{\frac{p}{2}}(0,T;\mathbb{R}^{d})$. Moreover, for any $q\in(2,p]$,
\[
\mathbb{E}\left[\sup_{0\leq t\leq T}|x_{\theta}(t)|^q+\sup_{0\leq t\leq T}|y_{\theta}(t)|^{\frac{q}{2}}+\left(\int^T_0|z_{\theta}(t)|^2dt\right)^{\frac{q}{4}}\right]\leq  C(L,T,q) \mathbb{E}\left[|x_0|^q+\int^T_0|u(t)|^qdt \right].
\]
\end{lemma}
\begin{proof}
The proof is immediate from Lemma \ref{myq1} and Lemma \ref{myq2} in  appendix A.
\end{proof}

\begin{lemma}\label{myw306} Assume that \emph{(H1)} and \emph{(H4)}  hold. Then,  $\theta\rightarrow{y}_{\theta}(0)$ is continuous and bounded.
\end{lemma}
\begin{proof}
The proof is immediate from Lemma \ref{myw301} and Lemma \ref{myw209} in appendix B.
\end{proof}

Suppose  the conditions {(H1)} and {(H4)}  hold.
It follows from Lemma \ref{myw306} that, $y_{\theta}(0)$ is continuous in $\theta$ and  $J(u)$ is well-defined. Then, our  stochastic optimal control problem is to minimize the robust cost $J(u)$ over  $u\in\mathcal{U}[0,T]$.

\section{Stochastic maximum principle}

In this section, we will establish the stochastic maximum principle by the linearization and weak convergence methods, which is different from the classical variational approach due to the model uncertainty.

\subsection{Variational equation}
Let $\overline{u}\in\mathcal{U}[0,T]$ be an optimal control and $(\overline{x}_{\theta},\overline{y}_{\theta},\overline{z}_{\theta})$ be the corresponding state process of equations \eqref{App1y} and \eqref{App1y11}  for each $\theta\in\Theta$. Note that the set $\mathcal{U}[0,T]$ is convex. Then, for any $u\in\mathcal{U}[0,T]$ and $\rho\in(0,1)$, it is easy to check that the process $u^{\rho}:=\overline{u}+\rho(u-\overline{u})$ is also an admissible control. Denote by $(x_{\theta}^{\rho},y_{\theta}^{\rho},z_{\theta}^{\rho})$ the trajectory corresponding to $u^{\rho}$ for any  $\theta\in\Theta$.

First, we introduce the following variational  SDE on the time interval $[0,T]$: for each  $\theta\in\Theta$,
\begin{align}\label{myw101}
\begin{cases}
&d\widehat{x}_{\theta}(t)=\left(\partial_x b_{\theta}(t)\widehat{x}_{\theta}(t)+\partial_u b_{\theta}(t)(u(t)-\overline{u}(t))\right)dt+
\sum\limits_{i=1}^d\left(\partial_x \sigma^i_{\theta}(t)\widehat{x}_{\theta}(t)+\partial_u \sigma^i_{\theta}(t)(u(t)-\overline{u}(t))\right)dW^i(t),\\
&\widehat{x}_{\theta}(0)=0,
\end{cases}
\end{align}
where $b_{\theta}(t)=b_{\theta}(t,\overline{x}_{\theta}(t),\overline{u}(t))$, $\partial_xb_{\theta}(t)=\partial_xb_{\theta}(t,\overline{x}_{\theta}(t),\overline{u}(t))$ and $\sigma^i_{\theta}(t)$, $\partial_x\sigma^i_{\theta}(t)$, $\partial_u b_{\theta}(t)$, $\partial_u\sigma^i_{\theta}(t)$ are defined in a similar way.
It follows from assumption (H1) that $\partial_x b_{\theta},\partial_u b_{\theta}, \partial_x \sigma^i_{\theta}$ and $\partial_u \sigma^i_{\theta}$
are uniformly bounded. Then, from Lemma \ref{myq1} in appendix A, the SDE \eqref{myw101} admits a unique solution $\widehat{x}_{\theta}\in \mathcal{S}^p(0,T;\mathbb{R}^n)$. Moreover, it holds that
\begin{align}\label{myw102}
\mathbb{E}\left[\sup_{0\leq t\leq T}|\widehat{x}_{\theta}(t)|^q\right]\leq C(L,T,q)\mathbb{E}\left[\int^T_0(|u(t)|^q+|\overline{u}(t)|^q)dt\right], \ \forall q\in[2,p].
\end{align}

\begin{lemma} \label{myw201} Assume that \emph{(H1)}-\emph{(H3)}   hold. Then, for each $\theta\in\Theta$,
\begin{description}
\item[(i)] $\mathbb{E}\left[\sup_{0\leq t\leq T}|\widetilde {x}^{\rho}_{\theta}(t)|^4\right]\leq C(L,T)\mathbb{E}\left[\int^T_0(|u(t)|^4+|\overline{u}(t)|^4)dt\right]$,
\item[(ii)]
$
\lim\limits_{\rho\rightarrow 0}\sup\limits_{\theta\in\Theta}\mathbb{E}\left[\sup_{0\leq t\leq T}|\widetilde{x}^{\rho}_{\theta}(t)|^4\right]=0,
$
where $
\widetilde{x}_{\theta}^{\rho}(t):=\rho^{-1}\left(x_{\theta}^{\rho}(t)- \overline{x}_{\theta}(t)\right)-\widehat{x}_{\theta}(t).
$\end{description}
\end{lemma}
\begin{proof}
By the definition of $\widetilde{x}^{\rho}_{\theta}$, we obtain that
\begin{align*}
\begin{cases}
&d\widetilde {x}^{\rho}_{\theta}(t)=\rho^{-1}\left\{ b_{\theta}^{\rho}(t)-b_{\theta}(t)-\left[\rho\partial_x b_{\theta}(t)\widehat{x}_{\theta}(t)+\rho\partial_u b_{\theta}(s)(u(t)-\overline{u}(t))\right] \right\}dt\\
&\ \ \ \ \ \ \ \ \ \  +
\sum\limits_{i=1}^d\rho^{-1}\left\{ \sigma_{\theta}^{\rho,i}(t)-\sigma^i_{\theta}(t)-\left[\rho\partial_x \sigma^i_{\theta}(t)\widehat{x}_{\theta}(t)+\rho\partial_u \sigma^i_{\theta}(s)(u(t)-\overline{u}(t))\right] \right\}dW^i(t),\\
&\widetilde{x}^{\rho}_{\theta}(0)=0,
\end{cases}
\end{align*}
where $b_{\theta}^{\rho}(t)=b_{\theta}(t,x^{\rho}_{\theta}(t),u^{\rho}(t))$ and $\sigma_{\theta}^{\rho,i}(t)=\sigma^i_{\theta}(t,x^{\rho}_{\theta}(t),u^{\rho}(t))$.
For convenience, set
\begin{align*}
 &A_{\theta}^{\rho}(t)=\int^1_0 \partial_x b_{\theta}(t,\overline{x}_{\theta}(t)+\lambda\rho(\widetilde{x}_{\theta}^{\rho}(t)+\widehat{x}_{\theta}(t)), \overline{u}(t)+\lambda\rho(u(t)-\overline{u}(t)))d\lambda,
 \\
 &B_{\theta}^{\rho,i}(t)=\int^1_0 \partial_x \sigma^i_{\theta}(t,\overline{x}_{\theta}(t)+\lambda\rho(\widetilde{x}_{\theta}^{\rho}(t)+\widehat{x}_{\theta}(t)), \overline{u}(t)+\lambda\rho(u(t)-\overline{u}(t)))d\lambda,\\
 &C_{\theta}^{\rho}(t)= \int^1_0 \left[\partial_u b_{\theta}(t,\overline{x}_{\theta}(t)+\lambda\rho(\widetilde{x}_{\theta}^{\rho}(t)+\widehat{x}_{\theta}(t)), \overline{u}(t)+\lambda\rho(u(t)-\overline{u}(t)))- \partial_u b_{\theta}(t)\right](u(t)-\overline{u}(t))d\lambda\\
 &\ \ \ \ \ \ \ \ \ \ \ \ \ \ +\left[A_{\theta}^{\rho}(t)- \partial_x b_{\theta}(t)\right]\widehat{x}_{\theta}(t),\\
 &D_{\theta}^{\rho,i}(t)= \int^1_0 \left[\partial_u \sigma^i_{\theta}(t,\overline{x}_{\theta}(t)+\lambda\rho(\widetilde{x}_{\theta}^{\rho}(t)+\widehat{x}_{\theta}(t)), \overline{u}(t)+\lambda\rho(u(t)-\overline{u}(t)))- \partial_u \sigma^i_{\theta}(t)\right](u(t)-\overline{u}(t))d\lambda\\
 &\ \ \ \ \ \ \ \ \ \ \ \ \ \ +\left[B_{\theta}^{\rho,i}(t)- \partial_x \sigma^i_{\theta}(t)\right]\widehat{x}_{\theta}(t).
\end{align*}
Thus, the process $\widetilde{x}^{\rho}_{\theta}$ could be regarded  as the solution to the following SDE:
\begin{align*}
\widetilde {x}^{\rho}_{\theta}(t)=\int^t_0\left(A_{\theta}^{\rho}(s)\widetilde {x}^{\rho}_{\theta}(s)+C_{\theta}^{\rho}(s)\right)ds+\sum\limits_{i=1}^d\int^t_0\left(B_{\theta}^{\rho,i}(s)\widetilde {x}^{\rho}_{\theta}(s)+D_{\theta}^{\rho,i}(s)\right)dW^i(s).
\end{align*}
Note that $\partial_x b_{\theta},\partial_u b_{\theta}, \partial_x \sigma^i_{\theta}$ and $\partial_u \sigma^i_{\theta}$
are bounded by some constant $C(L)$. Then, applying Lemma \ref{myq1}  in appendix A and inequality \eqref{myw102} yields that
\begin{align}\label{myw1041}
\begin{split}
\mathbb{E}\left[\sup_{0\leq t\leq T}|\widetilde {x}^{\rho}_{\theta}(t)|^4\right]&\leq C(L,T)\mathbb{E}\left[\int^T_0\left(|C_{\theta}^{\rho}(s)|^4+\sum\limits_{i=1}^d|D_{\theta}^{\rho,i}(s)|^4\right)ds  \right] \\
&\leq C(L,T)\mathbb{E}\left[\int^T_0(|u(t)|^4+|\overline{u}(t)|^4)dt\right],
\end{split}
\end{align}
which establishes the first inequality.

Next, we prove the term (ii). It suffices to show that
\[
\lim\limits_{\rho\rightarrow 0}\sup\limits_{\theta\in\Theta}\mathbb{E}\left[\int^T_0|C_{\theta}^{\rho}(t)|^4dt \right]=0,
\]
since the other case could be proved in a similar fashion.  According to H\"{o}lder's inequality, we get that
\begin{align*}
|C_{\theta}^{\rho}(t)|^4\leq& 8\int^1_0|\partial_u b_{\theta}(t,\overline{x}_{\theta}(t)+\lambda\rho(\widetilde{x}_{\theta}^{\rho}(t)+\widehat{x}_{\theta}(t)), \overline{u}(t)+\lambda\rho(u(t)-\overline{u}(t)))- \partial_u b_{\theta}(t)|^4|u(t)-\overline{u}(t)|^4d\lambda\\
&+8\int^1_0|\partial_x b_{\theta}(t,\overline{x}_{\theta}(t)+\lambda\rho(\widetilde{x}_{\theta}^{\rho}(t)+\widehat{x}_{\theta}(t)), \overline{u}(t)+\lambda\rho(u(t)-\overline{u}(t)))- \partial_x b_{\theta}(t)|^4|\widehat{x}_{\theta}(t)|^4d\lambda.
\end{align*}
From assumption (H3), we derive that,
\begin{align*}
&\left|\partial_{v} b_{\theta}(t,\overline{x}_{\theta}(t)+\lambda\rho(\widetilde{x}_{\theta}^{\rho}(t)+\widehat{x}_{\theta}(t)), \overline{u}(t)+\lambda\rho(u(t)-\overline{u}(t)))- \partial_{v} b_{\theta}(t)\right|\leq \overline{\omega}(2N\rho),\ \text{for $v=u,x$},
\end{align*}
whenever $|\widetilde{x}_{\theta}^{\rho}(t)+\widehat{x}_{\theta}(t)|\leq N$ and $|u(t)-\overline{u}(t)|\leq N$ for each $N>0$.
Since $\partial_xb_{\theta}$ and $\partial_ub_{\theta}$ are bounded by some constant $C(L)$,  it holds that
\begin{align}\label{myw5045}
\begin{split}
&\mathbb{E}\left[\int^T_0|C_{\theta}^{\rho}(t)|^4dt\right]\leq 8(\overline{\omega}(2N\rho))^4\mathbb{E}\left[\int^T_0 \left(|u(t)-\overline{u}(t)|^4+|\widehat{x}_{\theta}(t)|^4\right)dt\right]\\
&\ \ \ \ \ +C(L)\mathbb{E}\left[\int^T_0 \left(|u(t)-\overline{u}(t)|^4+|\widehat{x}_{\theta}(t)|^4\right)\left(I_{\{|\widetilde{x}_{\theta}^{\rho}(t)+\widehat{x}_{\theta}(t)|\geq N\}}+I_{\{|u(t)-\overline{u}(t)|\geq N\}}\right)dt\right].
\end{split}
\end{align}

On the other hand, applying H\"{o}lder's inequality  yields that
\begin{align*}
\mathbb{E}\left[\int^T_0 |u(t)-\overline{u}(t)|^4I_{\{|\widetilde{x}_{\theta}^{\rho}(t)+\widehat{x}_{\theta}(t)|\geq N\}}dt\right]
&\leq \mathbb{E}\left[\int^T_0 |u(t)-\overline{u}(t)|^pdt\right]^{\frac{4}{p}}\mathbb{E}\left[\int^T_0I_{\{|\widetilde{x}_{\theta}^{\rho}(t)+\widehat{x}_{\theta}(t)|\geq N\}}dt\right]^{\frac{p-4}{p}}\\
&\leq \mathbb{E}\left[\int^T_0 |u(t)-\overline{u}(t)|^pdt\right]^{\frac{4}{p}}\mathbb{E}\left[\int^T_0\frac{|\widetilde{x}_{\theta}^{\rho}(t)+\widehat{x}_{\theta}(t)|} {N}dt\right]^{\frac{p-4}{p}} \\ &\leq C(L,T,p)\bigg(1+\mathbb{E}\bigg[\int^T_0(|u(t)|^p+|\overline{u}(t)|^p)dt\bigg]\bigg){N^{\frac{4-p}{p}}},
\end{align*}
where we have used estimates \eqref{myw102} and \eqref{myw1041} in the last inequality. By a similar analysis, we could also get that
\begin{align*}
\mathbb{E}\left[\int^T_0 \left(|u(t)-\overline{u}(t)|^4+|\widehat{x}_{\theta}(t)|^4\right)\left(I_{\{|\widetilde{x}_{\theta}^{\rho}(t)+\widehat{x}_{\theta}(t)|\geq N\}}+I_{\{|u(t)-\overline{u}(t)|\geq N\}}\right)dt\right] \leq C(L,T,u,\overline{u},p){N^{\frac{4-p}{p}}}.
\end{align*}

Consequently, with the help of inequality \eqref{myw5045}, we deduce that, for each $N>0$,
\begin{align*}
&\sup\limits_{\theta\in\Theta}\mathbb{E}\left[\int^T_0|C_{\theta}^{\rho}(t)|^4dt\right]\leq   C(L,T,u,\overline{u},p)\left(|\overline{\omega}(2N\rho)|^4+{N^{\frac{4-p}{p}}}\right),
\end{align*}
Sending $\rho\rightarrow 0$ and then   $N\rightarrow\infty$, we could get the desired equation.
\end{proof}

Next, we consider the corresponding variational BSDE on $[0,T]$: for each $\theta\in\Theta$,
\begin{align}\label{myw106}
\begin{split}
\widehat{y}_{\theta}(t)=&\partial_x\varphi_{\theta}(\overline{x}_{\theta}(T))\widehat{x}_{\theta}(T)-\int^T_t\widehat{z}_{\theta}(s)dW(s)\\
&+\int^T_t\left[\partial_xf_{\theta}(s)\widehat{x}_{\theta}(s)+\partial_yf_{\theta}(s)\widehat{y}_{\theta}(s)+\partial_zf_{\theta}(s)\left(\widehat{z}_{\theta}(s)\right)^{\top}+\partial_uf_{\theta}(s)(u(s)-\overline{u}(s))\right]ds,
\end{split}
\end{align}
where $f_{\theta}(t)=f_{\theta}(t,\overline{x}_{\theta}(t),\overline{y}_{\theta}(t),\overline{z}_{\theta}(t),\overline{u}(t))$ and $\partial_xf_{\theta}(t),\partial_yf_{\theta}(t),\partial_zf_{\theta}(t),\partial_uf_{\theta}(t)$ are defined in a similar way.
It follows from assumption (H1) that $\partial_y f_{\theta}$, $\partial_z f_{\theta}$
are uniformly bounded and $\partial_x\varphi_{\theta}, \partial_x f_{\theta},\partial_u f_{\theta}$ are bounded by $C(L)(1+|x|+|u|)$. Then, from Lemma \ref{myq2} in appendix A, the BSDE \eqref{myw106} admits a unique solution $(\widehat{y}_{\theta},\widehat{z}_{\theta})\in \mathcal{S}^{\frac{p}{2}}(0,T;\mathbb{R}^n)\times\mathcal{H}^{\frac{p}{2}}(0,T;\mathbb{R}^n)$. Moreover, it holds that, for each $ q\in[2,\frac{p}{2}]$,
\begin{align}\label{myw107}
\mathbb{E}\left[\sup_{0\leq t\leq T}|\widehat{y}_{\theta}(t)|^q+\left(\int^T_0|\widehat{z}_{\theta}(t)|^2dt\right)^{\frac{q}{2}}\right]\leq C(L,T,q)\mathbb{E}\left[|x_0|^{2q}+\int^T_0(|u(t)|^{2q}+|\overline{u}(t)|^{2q})dt\right].
\end{align}
\begin{lemma}\label{myw202} Assume that \emph{(H1)}-\emph{(H3)}  hold. Then, for each $\theta\in\Theta$,
\begin{description}
\item[(i)] $\mathbb{E}\left[\sup_{0\leq t\leq T}|\widetilde{y}^{\rho}_{\theta}(t)|^2+\int^T_0 |\widetilde{z}^{\rho}_{\theta}(t)|^2dt\right]\leq C(L,T)\mathbb{E}\left[|x_0|^4+\int^T_0(|u(t)|^{4}+|\overline{u}(t)|^{4})dt\right]$,
\item[(ii)] $\lim\limits_{\rho\rightarrow 0}\sup\limits_{\theta\in\Theta}\mathbb{E}\left[\sup_{0\leq t\leq T}|\widetilde{y}^{\rho}_{\theta}(t)|^2+\int^T_0 |\widetilde{z}^{\rho}_{\theta}(t)|^2dt\right]=0$,
\end{description}
where $
\widetilde{y}_{\theta}^{\rho}(t):=\rho^{-1}\left(y_{\theta}^{\rho}(t)- \overline{y}_{\theta}(t)\right)-\widehat{y}_{\theta}(t)
$ and $\widetilde{z}_{\theta}^{\rho}(t):=\rho^{-1}\left(z_{\theta}^{\rho}(t)- \overline{z}_{\theta}(t)\right)-\widehat{z}_{\theta}(t)$.
\end{lemma}
\begin{proof}
By the definition of $\widetilde{y}^{\rho}_{\theta}$ and $\widetilde{z}^{\rho}_{\theta}$, we obtain that
\begin{align*}
&\widetilde {y}^{\rho}_{\theta}(t)=\rho^{-1}\left[\varphi_{\theta}(x^{\rho}_{\theta}(T))-\varphi_{\theta}(\overline{x}_{\theta}(T))-\rho\partial_x\varphi_{\theta}(\overline{x}_{\theta}(T))\widehat{x}_{\theta}(T)\right]-\int^T_t\widetilde{z}^{\rho}_{\theta}(s)dW(s)\\
&\ \ +\int^T_t\left[\rho^{-1}(f^{\rho}_{\theta}(t)-f_{\theta}(t))- \partial_xf_{\theta}(s)\widehat{x}_{\theta}(s)-\partial_yf_{\theta}(s)\widehat{y}_{\theta}(s)-\partial_zf_{\theta}(s)\left(\widehat{z}_{\theta}(s)\right)^{\top}
-\partial_uf_{\theta}(s)(u(s)-\overline{u}(s))\right]ds,
\end{align*}
where $f^{\rho}_{\theta}(t)=f_{\theta}(t,x^{\rho}_{\theta}(t),y^{\rho}_{\theta}(t),z^{\rho}_{\theta}(t),u^{\rho}(t))$.
To simplify symbols,  set $\gamma=(x,y,z)$ and
\begin{align*}
&J^{1,\rho}_{\theta}=\int^1_0\partial_x\varphi_{\theta}(\overline{x}_{\theta}(T)+\lambda\rho(\widetilde{x}^{\rho}_{\theta}(T)+\widehat{x}_{\theta}(T)))d\lambda\widetilde{x}^{\rho}_{\theta}(T),\\
&  J^{2,\rho}_{\theta}=\int^1_0\left[\partial_x\varphi_{\theta}(\overline{x}_{\theta}(T)+\lambda\rho(\widetilde{x}^{\rho}_{\theta}(T)+\widehat{x}_{\theta}(T)))- \partial_x\varphi_{\theta}(\overline{x}_{\theta}(T))\right]d\lambda\widehat{x}_{\theta}(T),\\
&E_{\theta}^{\rho}(t)=\int^1_0 \partial_x f_{\theta}(t,\overline{\gamma}_{\theta}(t)+\lambda\rho(\widetilde{\gamma}_{\theta}^{\rho}(t)+\widehat{\gamma}_{\theta}(t)) ,\overline{u}(t)+\lambda\rho(u(t)-\overline{u}(t)))d\lambda,
 \\
 &F_{\theta}^{\rho}(t)=\int^1_0 \partial_y f_{\theta}(t,\overline{\gamma}_{\theta}(t)+\lambda\rho(\widetilde{\gamma}_{\theta}^{\rho}(t)+\widehat{\gamma}_{\theta}(t)) ,\overline{u}(t)+\lambda\rho(u(t)-\overline{u}(t)))d\lambda,\\
 &G_{\theta}^{\rho}(t)= \int^1_0 \partial_z f_{\theta}(t,\overline{\gamma}_{\theta}(t)+\lambda\rho(\widetilde{\gamma}_{\theta}^{\rho}(t)+\widehat{\gamma}_{\theta}(t)) ,\overline{u}(t)+\lambda\rho(u(t)-\overline{u}(t)))d\lambda,\\
 &H_{\theta}^{\rho}(t)= \int^1_0 \left[\partial_u f_{\theta}(t,\overline{\gamma}_{\theta}(t)+\lambda\rho(\widetilde{\gamma}_{\theta}^{\rho}(t)+\widehat{\gamma}_{\theta}(t)) ,\overline{u}(t)+\lambda\rho(u(t)-\overline{u}(t)))- \partial_u f_{\theta}(t)\right](u(t)-\overline{u}(t))d\lambda\\
 &\ \ \ \ \ \ \ \ \ \ \ \ \ \ +\left[E_{\theta}^{\rho}(t)- \partial_x f_{\theta}(t)\right]\widehat{x}_{\theta}(t)+\left[F_{\theta}^{\rho}(t)- \partial_y f_{\theta}(t)\right]\widehat{y}_{\theta}(t)+\left[G_{\theta}^{\rho}(t)- \partial_z f_{\theta}(t)\right](\widehat{z}_{\theta}(t))^{\top}.
\end{align*}
Thus, the pair of processes $(\widetilde{y}^{\rho}_{\theta},\widetilde{z}^{\rho}_{\theta})$ satisfies the following BSDE on $[0,T]$:
\begin{align*}
\widetilde {y}^{\rho}_{\theta}(t)=J^{1,\rho}_{\theta}+J^{2,\rho}_{\theta}+\int^T_t\left(E_{\theta}^{\rho}(s)\widetilde {x}^{\rho}_{\theta}(s)+F_{\theta}^{\rho}(s)\widetilde {y}^{\rho}_{\theta}(s)+G_{\theta}^{\rho}(s)(\widetilde {z}^{\rho}_{\theta}(s))^{\top}+H_{\theta}^{\rho}(s)\right)ds-\int^T_t\widetilde {z}^{\rho}_{\theta}(s)dW(s).
\end{align*}
Applying Lemma \ref{myq2}  in appendix  A  yields that
\begin{align}\label{myw104}
\mathbb{E}\left[\sup_{0\leq t\leq T}|\widetilde{y}^{\rho}_{\theta}(t)|^2+\int^T_0 |\widetilde{z}^{\rho}_{\theta}(t)|^2dt\right]\leq C(L,T)\mathbb{E}\left[\left|J^{1,\rho}_{\theta}\right|^2+\left|J^{2,\rho}_{\theta}\right|^2+\left|\int^T_0\left(|E_{\theta}^{\rho}(t)\widetilde {x}^{\rho}_{\theta}(t)|+|H_{\theta}^{\rho}(t)|\right)dt\right|^2\right].
\end{align}

Recalling  assumption (H1) and the fact that $\rho(\widetilde{x}^{\rho}_{\theta}(t)+\widehat{x}_{\theta}(t))={x}^{\rho}_{\theta}(t)-\overline{x}_{\theta}(t)$, we could obtain that
\begin{align*}
&|J^{1,\rho}_{\theta}|\leq C(L)(1+|\overline{x}_{\theta}(T)|+|{x}^{\rho}_{\theta}(T)|)|\widetilde{x}^{\rho}_{\theta}(T)|,\
|J^{2,\rho}_{\theta}|\leq C(L)(1+|\overline{x}_{\theta}(T)|+|{x}^{\rho}_{\theta}(T)|)|\widehat{x}_{\theta}(T)|,\\
&|E_{\theta}^{\rho}(t)|\leq C(L)(1+|\overline{x}_{\theta}(t)|+|{x}^{\rho}_{\theta}(t)|+|u(t)|+|\overline{u}(t)|),\\
&|H_{\theta}^{\rho}(t)|\leq  C(L)(1+|\overline{x}_{\theta}(t)|^2+|{x}^{\rho}_{\theta}(t)|^2+|\widehat{x}_{\theta}(t)|^2+|u(t)|^2+|\overline{u}(t)|^2+|\widehat{y}^{\rho}_{\theta}(t)|+|\widehat {z}^{\rho}_{\theta}(t)|),
\end{align*}
which  together with Lemma \ref{myw301}, inequalities \eqref{myw102}, \eqref{myw1041}, \eqref{myw107} and \eqref{myw104} indicates that
\begin{align}\label{myw203}
\mathbb{E}\left[\sup_{0\leq t\leq T}|\widetilde{y}^{\rho}_{\theta}(t)|^2+\int^T_0 |\widetilde{z}^{\rho}_{\theta}(t)|^2dt\right]\leq C(L,T)\mathbb{E}\left[|x_0|^4+\int^T_0(|u(t)|^{4}+|\overline{u}(t)|^{4})dt\right].
\end{align}

Now, we are going to prove that the right side of inequality \eqref{myw104} converges to $0$ uniformly as $\rho\rightarrow 0$. The remainder of the proof will be given in  the following three steps.

{\bf Step 1 ($J^{1,\rho}_{\theta}+J^{2,\rho}_{\theta}$-term).}
By H\"{o}lder's inequality and Lemma \ref{myw301}, Lemma \ref{myw201}, it holds that
\begin{align*}
&\lim\limits_{\rho\rightarrow 0}\sup\limits_{\theta\in\Theta}\mathbb{E}\left[\left|J^{1,\rho}_{\theta}\right|^2\right]\leq C(L,T)\lim\limits_{\rho\rightarrow 0} \sup\limits_{\theta\in\Theta}\mathbb{E}\left[1+\left|\overline{x}_{\theta}(T)\right|^4+\left|{x}^{\rho}_{\theta}(T)\right|^4\right]^{\frac{1}{2}}\mathbb{E}\left[\left|\widetilde{x}^{\rho}_{\theta}(T)\right|^4\right]^{\frac{1}{2}}=0.
\end{align*}

On the other hand, by a similar analysis as  Lemma \ref{myw201}, we have that for each $N>0$,
\begin{align*}
&|J^{2,\rho}_{\theta}|
\leq\overline{\omega}(N\rho)|\widehat{x}_{\theta}(T)|+C(L)\left(1+|\overline{x}_{\theta}(T)|^2+|{x}^{\rho}_{\theta}(T)|^2+|\widehat{x}_{\theta}(T)|^2\right)I_{\{|\widetilde{x}_{\theta}^{\rho}(T)+\widehat{x}_{\theta}(T)|\geq N\}}.
\end{align*}
Therefore,  we deduce that
\begin{align*}
&\mathbb{E}\left[\left|J^{2,\rho}_{\theta}\right|^2\right]\leq 2(\overline{\omega}(N\rho))^2\mathbb{E}\left[|\widehat{x}_{\theta}(T)|^2\right]
+C(L)\mathbb{E}\left[\left(1+|\overline{x}_{\theta}(T)|^4+|{x}^{\rho}_{\theta}(T)|^4+|\widehat{x}_{\theta}(T)|^4\right)I_{\{|\widetilde{x}_{\theta}^{\rho}(T)+\widehat{x}_{\theta}(T)|\geq N\}}\right].
\end{align*}
With the help of H\"{o}lder's inequality,  we conclude that
\begin{align*}
&\mathbb{E}\left[\left(1+|\overline{x}_{\theta}(T)|^4+|{x}^{\rho}_{\theta}(T)|^4+|\widehat{x}_{\theta}(T)|^4\right)I_{\{|\widetilde{x}_{\theta}^{\rho}(T)+\widehat{x}_{\theta}(T)|\geq N\}}\right]\\
&\leq C \mathbb{E}\left[\left(1+|\overline{x}_{\theta}(T)|^p+|{x}^{\rho}_{\theta}(T)|^p+|\widehat{x}_{\theta}(T)|^p\right)\right]^{\frac{4}{p}}\mathbb{E}\left[\frac{|\widetilde{x}_{\theta}^{\rho}(T)+\widehat{x}_{\theta}(T)|} {N}\right]^{\frac{p-4}{p}}
\\
& \leq C(L,T,x_0,p)\bigg(1+\mathbb{E}\bigg[\int^T_0(|u(t)|^p+|\overline{u}(t)|^p)dt\bigg]\bigg){N^{\frac{4-p}{p}}}.
\end{align*}
It follows that, for each $N>0$,
\begin{align*}
&\sup\limits_{\theta\in\Theta}\mathbb{E}\left[\left|J^{2,\rho}_{\theta}\right|^2\right]\leq   C(L,T,x_0,p)\bigg(1+\mathbb{E}\bigg[\int^T_0(|u(t)|^p+|\overline{u}(t)|^p)dt\bigg]\bigg)\left( |\overline{\omega}(N\rho)|^2+{N^{\frac{4-p}{p}}}\right).
\end{align*}
Sending $\rho\rightarrow 0$ and letting $N\rightarrow\infty$, we could get the desired equation.

{\bf Step 2 ($E_{\theta}^{\rho}(t)\widetilde {x}^{\rho}_{\theta}(t)$-term).} Using  H\"{o}lder's inequality  and Lemma \ref{myw301} again, we conclude that
\begin{align*}
\mathbb{E}\left[\left(\int^T_0|E_{\theta}^{\rho}(t)\widetilde {x}^{\rho}_{\theta}(t)|dt\right)^2\right]\leq &
\mathbb{E}\left[\int^T_0|E_{\theta}^{\rho}(t)|^4dt\right]^{\frac{1}{2}} \mathbb{E}\left[\int^T_0\left|\widetilde{x}^{\rho}_{\theta}(t)\right|^4dt\right]^{\frac{1}{2}}\\ \leq& C(L,T,x_0)\bigg(1+\mathbb{E}\bigg[\int^T_0(|u(t)|^4+|\overline{u}(t)|^4)dt\bigg]\bigg)\mathbb{E}\left[\sup\limits_{t\in[0,T]}\left|\widetilde{x}^{\rho}_{\theta}(t)\right|^4\right]^{\frac{1}{2}},
\end{align*}
which together with  Lemma \ref{myw201} implies the desired equation holds.

{\bf Step 3 ($H^{\rho}_{\theta}$-term).}
From assumption (H3), it holds that
\begin{align*}
\left|\partial_{v} f_{\theta}(t,\overline{\gamma}_{\theta}(t)+\lambda\rho(\widetilde{\gamma}_{\theta}^{\rho}(t)+\widehat{\gamma}_{\theta}(t)), \overline{u}(t)+\lambda\rho(u(t)-\overline{u}(t)))- \partial_{v} f_{\theta}(t)\right|\leq \overline{\omega}(4N\rho),\ \text{for $v=u,x,y,z$},
	\end{align*}	
whenever $|\widetilde{x}_{\theta}^{\rho}(t)+\widehat{x}_{\theta}(t)|\leq N$, $|\widetilde{y}_{\theta}^{\rho}(t)+\widehat{y}_{\theta}(t)|\leq N$, $|\widetilde{z}_{\theta}^{\rho}(t)+\widehat{z}_{\theta}(t)|\leq N$  and $|u(t)-\overline{u}(t)|\leq N$ for each $N>0$.
 Thus, by a similar analysis as step 1,  we derive that
\begin{align*}
&\mathbb{E}\left[\left(\int^T_0|\left[E_{\theta}^{\rho}(t)- \partial_x f_{\theta}(t)\right]\widehat{x}_{\theta}(t)|dt\right)^2\right]\leq C(L)\mathbb{E}\bigg[\int^T_0|\overline{\omega}(4N\rho)\widehat{x}_{\theta}(t)|^2dt\\
& +\bigg|\int^T_0 \left(1+|\Gamma(t)|^2\right)\left(I_{\{|\widetilde{x}_{\theta}^{\rho}(t)+\widehat{x}_{\theta}(t)|\geq N\}}+I_{\{|\widetilde{y}_{\theta}^{\rho}(t)+\widehat{y}_{\theta}(t)|\geq N\}}+I_{\{|\widetilde{z}_{\theta}^{\rho}(t)+\widehat{z}_{\theta}(t)|\geq N\}}+I_{\{|u(t)-\overline{u}(t)|\geq N\}}\right)dt\bigg|^2\bigg],
\end{align*}
where $\Gamma(t):=|\overline{x}_{\theta}(t)|+|\widehat{x}_{\theta}(t)|+|{x}^{\rho}_{\theta}(t)|+|u(t)|+|\overline{u}(t)|$.
By a direct computation, we have that
\begin{align*}
&\mathbb{E}\bigg[\bigg|\int^T_0 \left(1+|\Gamma(t)|^2\right)\left(I_{\{|\widetilde{x}_{\theta}^{\rho}(t)+\widehat{x}_{\theta}(t)|\geq N\}}+I_{\{|\widetilde{y}_{\theta}^{\rho}(t)+\widehat{y}_{\theta}(t)|\geq N\}}+I_{\{|\widetilde{z}_{\theta}^{\rho}(t)+\widehat{z}_{\theta}(t)|\geq N\}}+I_{\{|u(t)-\overline{u}(t)|\geq N\}}\right)dt\bigg|^2\bigg]\\
&  \leq  C(L,T,x_0,p)\bigg(1+\mathbb{E}\bigg[\int^T_0(|u(t)|^p+|\overline{u}(t)|^p)dt\bigg]\bigg)^{\frac{4}{p}}\mathbb{E}\left[\int^T_0\frac{|\widetilde{\gamma}_{\theta}^{\rho}(t)+\widehat{\gamma}_{\theta}(t)|+|u(t)-\overline{u}(t)|}{N}dt\right]^{\frac{p-4}{p}}\\&
 \leq  C(L,T,x_0,p)\bigg(1+\mathbb{E}\bigg[\int^T_0(|u(t)|^p+|\overline{u}(t)|^p)dt\bigg]\bigg) N^{\frac{4-p}{p}},
\end{align*}
where we have used estimates  \eqref{myw102}, \eqref{myw1041}, \eqref{myw107} and \eqref{myw203} in the last inequality.  As a result, we derive that for each $N>0$
\begin{align*}
&\sup\limits_{\theta\in\Theta}\mathbb{E}\left[\left(\int^T_0|\left[E_{\theta}^{\rho}(t)- \partial_x f_{\theta}(t)\right]\widehat{x}_{\theta}(t)|dt\right)^2\right]\\
&\leq
 C(L,T,x_0,p)\bigg(1+\mathbb{E}\bigg[\int^T_0(|u(t)|^p+|\overline{u}(t)|^p)dt\bigg]\bigg)\left(|\overline{\omega}(4N\rho)|^2+N^{\frac{4-p}{p}}\right).
\end{align*}

By a similar argument, we could also obtain that, for each $N>0$,
\begin{align*}
&\sup\limits_{\theta\in\Theta}\mathbb{E}\left[\left(\int^T_0|H_{\theta}^{\rho}(t)|dt\right)^2\right]\leq C(L,T,x_0,u,\overline{u},p)\left(|\overline{\omega}(4N\rho)|^2+N^{\frac{4-p}{p}}\right).
\end{align*}
Letting $\rho$ tend to $0$ and then $N$ tend to $\infty$, we could get the desired result.
\end{proof}

\begin{remark}{\upshape
Since our value function involves a family  of uncertainty parameters $\theta$, we establish the convergence results for variational SDE and BSDE uniformly with respect to $\theta$,  which is crucial for the main result; see Lemma \ref{myw208} below.}
\end{remark}

Finally, we are going to discuss the variational inequality. For this purpose, we introduce the following  subset of $\mathcal{Q}$:
 for each $u\in\mathcal{U}[0,T]$,
\[
\mathcal{Q}^{u}=\left\{Q\in\mathcal{Q}| J({u})=\int_{\Theta} {y}_{\theta}(0)Q(d\theta)\right\}.
\]
\begin{theorem}\label{myw210} Suppose that \emph{(H1)}-\emph{(H5)}  hold. Then, there exists a probability $\overline{Q}\in\mathcal{Q}^{\overline{u}}$ such that
\[
\inf\limits_{u\in\mathcal{U}[0,T]}\int_{\Theta}\widehat{y}_{\theta}(0)\overline{Q}(d\theta)\geq 0.
\]
\end{theorem}

In order to prove Theorem \ref{myw210}, we need the following lemmas.

\begin{lemma}\label{myw206}
Assume \emph{(H1)}, \emph{(H4)} and  \emph{(H5)} are satisfied. Then, the set $\mathcal{Q}^u$ is non-empty for each $u\in\mathcal{U}[0,T]$.
\end{lemma}

\begin{proof}
By the definition $J({u})$, there exists a sequence $Q^N\in\mathcal{Q}$ so that
\[
J({u})-\frac{1}{N}\leq \int_{\Theta}{y}_{\theta}(0)Q^N(d\theta)\leq J({u}).
\]
Note that $\mathcal{Q}$ is weakly compact. Then, choosing a subsequence if necessary, we could find a   ${Q}^u\in\mathcal{Q}$ such that $Q^N$ converges weakly  to ${Q}^u$. From   Lemma \ref{myw306}, the function $\theta\rightarrow{y}_{\theta}(0)$ is continuous and bounded. It follows that
\[
J({u})\geq \int_{\Theta} {y}_{\theta}(0){Q}^u(d\theta)=\lim\limits_{N\rightarrow\infty}\int_{\Theta}{y}_{\theta}(0)Q^N(d\theta)\geq J({u}),
\]
which ends the proof.
\end{proof}

\begin{lemma} \label{myw208} Assume that the conditions \emph{(H1)}-\emph{(H5)}  hold. Then, for each $u\in\mathcal{U}[0,T]$, there exists a probability $\overline{Q}\in\mathcal{Q}^{\overline{u}}$ so that
\[
\lim\limits_{\rho\rightarrow 0}\frac{J(u^{\rho})-J(\overline{u})}{\rho}= \sup\limits_{Q\in\mathcal{Q}^{\overline{u}}} \int_{\Theta} \widehat{y}_{\theta}(0) Q(d\theta)=\int_{\Theta} \widehat{y}_{\theta}(0) \overline{Q}(d\theta).
\]
\end{lemma}
\begin{proof} The proof is divided into the following two steps.

{\bf Step 1 (The convergence).}
For each ${Q}\in\mathcal{Q}^{\overline{u}}$, we have
\[
J(u^{\rho})\geq \int_{\Theta} y^{\rho}_{\theta}(0) {Q}(d\theta)\ \text{and} \ J(\overline{u})=\int_{\Theta} \overline{y}_{\theta}(0){Q}(d\theta),
\]
which implies that for each $\rho\in (0,1)$,
\begin{align*}
\frac{J(u^{\rho})-J(\overline{u})}{\rho}\geq \int_{\Theta} \frac{y^{\rho}_{\theta}(0)-\overline{y}_{\theta}(0)}{\rho} {Q}(d\theta)=
\int_{\Theta} (\widetilde{y}_{\theta}^{\rho}(0)+\widehat{y}_{\theta}(0)) {Q}(d\theta).
\end{align*}
On the other hand, by Lemma \ref{myw202}, we derive that
\[
\lim\limits_{\rho\rightarrow 0}\sup\limits_{\theta\in\Theta}|\widetilde{y}_{\theta}^{\rho}(0)|=0.
\] It follows that
\[
\liminf\limits_{\rho\rightarrow 0}\frac{J(u^{\rho})-J(\overline{u})}{\rho}\geq \int_{\Theta} \widehat{y}_{\theta}(0) {Q}(d\theta).
\]
As a result, we get that
\begin{align}\label{myw401}
\liminf\limits_{\rho\rightarrow 0}\frac{J(u^{\rho})-J(\overline{u})}{\rho}\geq \sup\limits_{Q\in\mathcal{Q}^{\overline{u}}} \int_{\Theta} \widehat{y}_{\theta}(0) Q(d\theta).
\end{align}

On the other hand, choosing a subsequence $\rho_N\rightarrow 0$ such that
\[
\limsup\limits_{\rho\rightarrow 0}\frac{J(u^{\rho})-J(\overline{u})}{\rho}=\lim\limits_{N\rightarrow \infty}\frac{J(u^{\rho_N})-J(\overline{u})}{\rho}.
\]
For each $N\geq 1$, recalling Lemma \ref{myw206}, we could find a probability $Q^{\rho_N}\in\mathcal{Q}^{{u}^{\rho_N}}$ so that
\[
J(u^{\rho_N})=\int_{\Theta} y^{\rho_N}_{\theta}(0) Q^{\rho_N}(d\theta)\ \text{and} \ J(\overline{u})\geq \int_{\Theta} \overline{y}_{\theta}(0) Q^{\rho_N}(d\theta),
\]
which indicates that
\begin{align*}
\frac{J(u^{\rho_N})-J(\overline{u})}{\rho}\leq \int_{\Theta} \frac{y^{\rho_N}_{\theta}(0)-\overline{y}_{\theta}(0)}{\rho} Q^{\rho_N}(d\theta)=
\int_{\Theta} (\widetilde{y}_{\theta}^{\rho_N}(0)+\widehat{y}_{\theta}(0)) Q^{\rho_N}(d\theta).
\end{align*}
 Choosing a subsequence if necessary,  there is a ${Q}^*\in\mathcal{Q}$ such that $(Q^{\rho_N})_{N\geq 1}$  converges weakly  to ${Q}^*$.
With the help of   Lemma \ref{myw2091} in appendix, the function
$\theta\rightarrow\widehat{y}_{\theta}(0)$ is continuous and bounded.
  Thus, by  Lemma \ref{myw202} and the property of weak convergence, we derive that
\begin{align}\label{myw402}
\limsup\limits_{\rho\rightarrow 0}\frac{J(u^{\rho})-J(\overline{u})}{\rho} \leq \lim\limits_{N\rightarrow \infty}\sup\limits_{\theta\in\Theta}|\widetilde{y}_{\theta}^{\rho_N}(0)|+\lim\limits_{N\rightarrow \infty}\int_{\Theta}\widehat{y}_{\theta}(0) Q^{\rho_N}(d\theta)=
\int_{\Theta} \widehat{y}_{\theta}(0) Q^{*}(d\theta).\end{align}
We claim that $Q^*\in\mathcal{Q}^{\overline{u}}$, which will be established in step 2.
Consequently, putting inequalities \eqref{myw401} and \eqref{myw402} together, we deduce that
\[
\lim\limits_{\rho\rightarrow 0}\frac{J(u^{\rho})-J(\overline{u})}{\rho}= \int_{\Theta} \widehat{y}_{\theta}(0) Q^{*}(d\theta)=\sup\limits_{Q\in\mathcal{Q}^{\overline{u}}} \int_{\Theta} \widehat{y}_{\theta}(0) Q(d\theta),
\]
which is the desired result.

{\bf Step 2 (The claim that $Q^*\in\mathcal{Q}^{\overline{u}}$).} Note that $y^{\rho}_{\theta}(t)-\overline{y}_{\theta}(t)=\rho(\widetilde{y}_{\theta}^{\rho}(t)+\widehat{y}_{\theta}(t))$. Then, with the help of inequalities \eqref{myw107} and \eqref{myw203},  we obtain that
\begin{align*}
\mathbb{E}\left[\sup\limits_{t\in[0,T]}|y^{\rho}_{\theta}(t)-\overline{y}_{\theta}(t)|^2\right]\leq C(L,T,x_0)\bigg(1+\mathbb{E}\bigg[\int^T_0(|u(t)|^4+|\overline{u}(t)|^4)dt\bigg]\bigg)\rho^2,
\end{align*}
which implies that
$
\lim\limits_{\rho\rightarrow 0}\sup\limits_{\theta\in\Theta}|y^{\rho}_{\theta}(0)-\overline{y}_{\theta}(0)|=0.
$
It follows from the definition of $J(u)$ that
\begin{align*}
&\lim\limits_{N\rightarrow \infty}|J(u^{\rho_N})-J(\overline{u})|\leq \lim\limits_{N\rightarrow \infty}\sup\limits_{\theta\in\Theta}|y^{\rho_N}_{\theta}(0)-\overline{y}_{\theta}(0)|=0,
\\
&\lim\limits_{N\rightarrow \infty}\int_{\Theta}|y^{\rho_N}_{\theta}(0)-\overline{y}_{\theta}(0)|Q^{\rho_N}(d\theta) \leq\lim\limits_{N\rightarrow \infty} \sup\limits_{\theta\in\Theta}|y^{\rho_N}_{\theta}(0)-\overline{y}_{\theta}(0)|=0.
\end{align*}
Consequently, we have that
\[
J(\overline{u})=\lim\limits_{N\rightarrow \infty}J(u^{\rho_N})=\lim\limits_{N\rightarrow \infty}\int_{\Theta}y^{\rho_N}_{\theta}(0)Q^{\rho_N}(d\theta)
=\lim\limits_{N\rightarrow \infty}\int_{\Theta}\overline{y}_{\theta}(0)Q^{\rho_N}(d\theta)=\int_{\Theta}\overline{y}_{\theta}(0)Q^{*}(d\theta),
\]
which completes the proof.
\end{proof}

Now, we are ready to complete the proof of Theorem \ref{myw210}.

\begin{proof}[Proof of  Theorem \ref{myw210}]
Denote by $\widehat{y}_{\theta}^{u}(0)$ the solution to  variational BSDE  \eqref{myw106} corresponding to the admissible control $u\in\mathcal{U}[0,T]$.
By Lemma \ref{myw208}, we obtain that,
\[
\lim\limits_{\rho\rightarrow 0}\frac{J(u^{\rho})-J(\overline{u})}{\rho}= \sup\limits_{Q\in\mathcal{Q}^{\overline{u}}} \int_{\Theta} \widehat{y}^{u}_{\theta}(0) Q(d\theta)\geq 0,
\]
which implies that
\[
\inf\limits_{u\in\mathcal{U}[0,T]}\sup\limits_{Q\in\mathcal{Q}^{\overline{u}}} \int_{\Theta} \widehat{y}^{u}_{\theta}(0) Q(d\theta)\geq 0.
\]

In the spirit of the fact that $\mathcal{Q}$ is  convex and weakly compact,  the subset $\mathcal{Q}^{\overline{u}}$ is also convex and weakly compact.
For each $\lambda\in[0,1], u,u^{\prime}\in\mathcal{U}[0,T]$, one can check that
\[
\widehat{y}^{\lambda u+(1-\lambda)u^{\prime}}_{\theta}(0)=\lambda \widehat{y}^{u}_{\theta}(0)+(1-\lambda)\widehat{y}^{u^{\prime}}_{\theta}(0).
\]
Moreover, with the help of  Lemma \ref{myq2}  in appendix  A, a direct computation yields that
\[
| \widehat{y}^{u}_{\theta}(0)- \widehat{y}^{u^{\prime}}_{\theta}(0)|\leq C(L,x_0,\overline{u})\mathbb{E}\left[\int^T_0|u(t)-u^{\prime}(t)|^4dt\right]^{\frac{1}{2}},
\]
from which we deduce that
$
u\rightarrow \int_{\Theta} \widehat{y}^{u}_{\theta}(0) Q(d\theta)\ \text{is continuous}.
$
It follows from Sion's minimax theorem that
\[
\inf\limits_{u\in\mathcal{U}[0,T]}\sup\limits_{Q\in\mathcal{Q}^{\overline{u}}} \int_{\Theta} \widehat{y}^{u}_{\theta}(0) Q(d\theta)=\sup\limits_{Q\in\mathcal{Q}^{\overline{u}}}\inf\limits_{u\in\mathcal{U}[0,T]} \int_{\Theta} \widehat{y}^{u}_{\theta}(0) Q(d\theta)\geq 0.
\]

For each $\varepsilon>0$, we can find a probability $Q^{\varepsilon} \in\mathcal{Q}^{\overline{u}}$ so that
\[
\inf\limits_{u\in\mathcal{U}[0,T]} \int_{\Theta} \widehat{y}^{u}_{\theta}(0) Q^{\varepsilon} (d\theta)\geq -\varepsilon.
\]
Since $\mathcal{Q}^{\overline{u}}$ is compact, there exists a subsequence $\varepsilon_n\rightarrow 0$ such that
$Q^{\varepsilon_n}$ converges weakly to some $\overline{Q}\in \mathcal{Q}^{\overline{u}}$. It follows that for each $u\in\mathcal{U}[0,T]$,
\[
\int_{\Theta} \widehat{y}^{u}_{\theta}(0) \overline{Q}(d\theta)=\lim\limits_{\varepsilon_n\rightarrow 0} \int_{\Theta} \widehat{y}^{u}_{\theta}(0) Q^{\varepsilon_n}(d\theta)\geq 0,
\]
which establishes the desired result.
\end{proof}

\begin{remark}
	{\upshape
Note that the Sion's minimax theorem is crucial for deriving the variational inequality with  a common probability $\overline{Q}\in\mathcal{Q}^{\overline{u}}$ for each $u\in\mathcal{U}[0,T]$. 
In order to use it, we assume that  the control domain is convex, and use the convex variation method  to ensure the solution $\widehat{y}^{u}_{\theta}$ of variational BSDE  is  convex in $u$. 
}
\end{remark}

\subsection{Maximum principle}
In this section, we will consider the necessary condition for the optimal control with the help the previous variational SDEs and BSDEs.

First, we will introduce the adjoint equation for the variational BSDE \eqref{myw106}. For this purpose, suppose that the solution to equation \eqref{myw106} satisfies that
\begin{align}\label{myw508}
\widehat{y}_{\theta}(t)=\left\langle p^1_{\theta}(t), \widehat{x}_{\theta}(t)\right\rangle+p_{\theta}^2(t),
\end{align}
where $ (p^1_{\theta}(t), p^2_{\theta}(t))$ is the solution to the following BSDE:
\begin{align}\label{myw509}
\begin{split}
&dp^1_{\theta}(t)=-P^1_{\theta}(t)dt+\sum\limits_{i=1}^dq_{\theta}^{1,i}(t)dW^i(t), \ \ p^1_{\theta}(T)=(\partial_x\varphi_{\theta}(\overline{x}_{\theta}(T)))^{\top},\\
&dp^2_{\theta}(t)=-P^2_{\theta}(t)dt+\sum\limits_{i=1}^dq_{\theta}^{2,i}(t)dW^i(t), \ \ p_{\theta}^2(T)=0.
\end{split}
\end{align}
Here  the functions $P^1_{\theta}$ and $P^2_{\theta}$ will be determined later.

Recalling equation \eqref{myw101} and applying It\^{o}'s formula to $\left\langle p_{\theta}^1(t), \widehat{x}_{\theta}(t)\right\rangle+p_{\theta}^2(t)$ yields that
\begin{align*}
&d\left(\left\langle p_{\theta}^1(t), \widehat{x}_{\theta}(t)\right\rangle+p_{\theta}^2(t)\right)\\
&=\left[\left\langle  p_{\theta}^1(t), \partial_x b_{\theta}(t)\widehat{x}_{\theta}(t)+\partial_u b_{\theta}(t)(u(t)-\overline{u}(t))\right\rangle+\sum\limits_{i=1}^d\left\langle q_{\theta}^{1,i},\partial_x \sigma^i_{\theta}(t)\widehat{x}_{\theta}(t)+\partial_u \sigma^i_{\theta}(t)(u(t)-\overline{u}(t))\right\rangle\right] dt\\
&\ \ \ \ +\sum\limits_{i=1}^d\left[
\left\langle p_{\theta}^1(t),\partial_x \sigma^i_{\theta}(t)\widehat{x}_{\theta}(t)+\partial_u \sigma^i_{\theta}(t)(u(t)-\overline{u}(t))\right\rangle+\left\langle q_{\theta}^{1,i},\widehat{x}_{\theta}(t)\right\rangle +q_{\theta}^{2,i}
\right]dW^i(t)\\
&\ \ \ \ -\left[\left\langle P_{\theta}^1(t),\widehat{x}_{\theta}(t)\right\rangle+P_{\theta}^2(t)\right] dt
\end{align*}
According to equation \eqref{myw106}, we get that
\begin{align*}
&\left\langle P_{\theta}^1(t),\widehat{x}_{\theta}(t)\right\rangle+P_{\theta}^2(t)=\partial_xf_{\theta}(t)\widehat{x}_{\theta}(t)+\partial_yf_{\theta}(t)\widehat{y}_{\theta}(t)+\sum\limits_{i=1}^d\partial_{z^i}f_{\theta}(t)\widehat{z}_{\theta}^i(t)+\partial_uf_{\theta}(t)(u(t)-\overline{u}(t))\\
&\ \ \ \ \  +\left\langle  p_{\theta}^1(t), \partial_x b_{\theta}(t)\widehat{x}_{\theta}(t)+\partial_u b_{\theta}(t)(u(t)-\overline{u}(t))\right\rangle+\sum\limits_{i=1}^d\left\langle q_{\theta}^{1,i},\partial_x \sigma^i_{\theta}(t)\widehat{x}_{\theta}(t)+\partial_u \sigma^i_{\theta}(t)(u(t)-\overline{u}(t))\right\rangle,\\
&\widehat{z}_{\theta}^i(t)=\left\langle p_{\theta}^1(t),\partial_x \sigma^i_{\theta}(t)\widehat{x}_{\theta}(t)+\partial_u \sigma^i_{\theta}(t)(u(t)-\overline{u}(t))\right\rangle+\left\langle q_{\theta}^{1,i},\widehat{x}_{\theta}(t)\right\rangle +q_{\theta}^{2,i},
\end{align*}
which together with equation \eqref{myw508} implies that
\begin{align*}
&P^1_{\theta}(t)=(\partial_xf_{\theta}(t))^{\top}+\left((\partial_x b_{\theta}(t))^{\top}+\partial_yf_{\theta}(t)+\sum\limits_{i=1}^d\partial_{z^i}f_{\theta}(t)(\partial_x\sigma^i_{\theta}(t))^{\top}\right)p^1_{\theta}(t)\\
&\ \ \ \ \ \ \ \ \ \ \ \ +\sum\limits_{i=1}^d\left(\partial_{z^i}f_{\theta}(t)+(\partial_x\sigma^i_{\theta}(t))^{\top}\right)q_{\theta}^{1,i}(t),\\
&P^2_{\theta}(t)=\left\langle p_{\theta}^1(t),\left(\partial_u b_{\theta}(t)+\sum\limits_{i=1}^d\partial_{z^i}f_{\theta}(t)\partial_u \sigma^i_{\theta}(t)\right)(u(t)-\overline{u}(t))\right\rangle+\sum\limits_{i=1}^d\left\langle q_{\theta}^{1,i}(t),\partial_u \sigma^i_{\theta}(t)(u(t)-\overline{u}(t))\right\rangle\\
&\ \ \ \ \ \ \ \ \ \ \ \ +\partial_uf_{\theta}(t)(u(t)-\overline{u}(t))+\partial_yf_{\theta}(t)p_{\theta}^2(t)+\sum\limits_{i=1}^d\partial_{z^i}f_{\theta}(t)q_{\theta}^{2,i}(t).
\end{align*}

From Lemma \ref{myq2} in appendix A,  the BSDE \eqref{myw509} admits a unique solution\[(p_{\theta}^1,q_{\theta}^1,p_{\theta}^2,q_{\theta}^2)\in \mathcal{S}^p(0,T;\mathbb{R}^n)\times \mathcal{H}^p(0,T;\mathbb{R}^{n\times d})\times \mathcal{S}^{\frac{p}{2}}(0,T;\mathbb{R})\times  \mathcal{H}^{\frac{p}{2}}(0,T;\mathbb{R}^{d}).\] Moreover, it holds that, for each $ q\in (2,p]$,
\begin{align}\label{myw907}
\mathbb{E}\left[\sup_{0\leq t\leq T}|p_{\theta}^1(t)|^q+\left(\int^T_0|q^1_{\theta}(t)|^2dt\right)^{\frac{q}{2}}\right]\leq C(L,T)\mathbb{E}\left[|x_0|^{q}+\int^T_0|\overline{u}(t)|^{q}dt\right].
\end{align}

Thus, we have the following.
\begin{lemma}\label{myw610}
Suppose \emph{(H1)}-\emph{(H4)} hold. Then, it holds that
\[
\widehat{y}_{\theta}(0)=p^2_{\theta}(0).
\]
\end{lemma}

Next, consider the following SDE:
\begin{align}\label{myw909}
dm_{\theta}(t)=\partial_yf_{\theta}(t)m_{\theta}(t)dt+\partial_{z}f_{\theta}(t)m_{\theta}(t)dW(t), \ \ m_{\theta}(0)=1.
\end{align}
Applying It\^{o}'s formula to $p^2_{\theta}(t)m_{\theta}(t)$ yields that
\begin{align*}
p^2_{\theta}(0)=\mathbb{E}\left[ \int^T_0\left\langle m_{\theta}(t)\partial_uH_{\theta}(t,\overline{x}_{\theta}(t),\overline{y}_{\theta}(t),\overline{z}_{\theta}(t),\overline{u}(t),\overline{u}(t),p^1_{\theta}(t),q^1_{\theta}(t)), u(t)-\overline{u}(t)\right\rangle dt
\right],
\end{align*}
where the Hamiltonian $H_{\theta}:[0,T]\times\mathbb{R}^n\times\mathbb{R}\times\mathbb{R}^d\times U\times U\times\mathbb{R}^n\times\mathbb{R}^{n\times d}\rightarrow \mathbb{R}$ is defined by
\begin{align*}
H_{\theta}(t,x,y,z,u,u^{\prime},p,q)=&\left\langle p, b_{\theta}(t,x,u)+\sum\limits_{i=1}^d\partial_{z^i}f_{\theta}(t,x,y,z,u^{\prime})\sigma^i_{\theta}(t,x,u)\right\rangle+\sum\limits_{i=1}^d\left\langle q^{i}, \sigma^i_{\theta}(t,x,u)\right\rangle \\
&\ \ +f_{\theta}(t,x,y,z,u).
\end{align*}
Recalling Theorem \ref{myw210} and Lemma \ref{myw610}, we conclude that for each $u\in\mathcal{U}[0,T]$,
\begin{align}\label{myw2094}
\int_{\Theta} \mathbb{E}\left[ \int^T_0\left\langle m_{\theta}(t)\partial_uH_{\theta}(t,\overline{x}_{\theta}(t),\overline{y}_{\theta}(t),\overline{z}_{\theta}(t),\overline{u}(t),\overline{u}(t),p^1_{\theta}(t),q^1_{\theta}(t)), u(t)-\overline{u}(t)\right\rangle dt
\right]\overline{Q}(d\theta)\geq 0.
\end{align}
Then, in order to derive a maximum principle, we need to study the measurability of the above integrand with respect to the argument $\theta$.

\begin{lemma}\label{myw20941}
Assume that \emph{(H1)}-\emph{(H4)} hold. Then,
the map $(\theta,t,\omega)\rightarrow \Pi_{\theta}(t,\omega)$ is a $\mathscr{F}$-progressively measurable process, i.e.,
for every $t\in[0,T]$, the function $\Pi_{\theta}(s,\omega):\Theta\times[0,t]\times\Omega\rightarrow\mathbb{R}$ is $\mathcal{B}(\Theta)\times\mathcal{B}([0,t])\times\mathscr{F}_t$-measurable, where
\begin{align}\label{myw2093}
\Pi_{\theta}(t)= m_{\theta}(t)\partial_uH_{\theta}(t,\overline{x}_{\theta}(t),\overline{y}_{\theta}(t),\overline{z}_{\theta}(t),\overline{u}(t),\overline{u}(t),p^1_{\theta}(t),q^1_{\theta}(t)).
\end{align}
\end{lemma}
\begin{proof}
	Note that  $\Theta$ is a Polish space.
For each $N>1$, choose a compact subset $K^N\subset\Theta$ satisfying that $\overline{Q}(\theta\notin K^N)\leq \frac{1}{N}.$
Then, we  could find a sequence of open neighborhoods $\left(B\left(\theta_l,\frac{1}{2N}\right)\right)_{l=1}^{L_N}$ so that
$K^N\subset\cup_{l=1}^{L_N}B\left(\theta_l,\frac{1}{2N}\right)$. Since   $\Theta$  is locally compact, by partitions of unity, there is a sequence of continuous functions $\eta_l:\Theta\rightarrow\mathbb{R}$ taking value in $[0,1]$ such that
\[
\eta_l(\theta)=0,  \ \text{for $\theta\notin B\left(\theta_l,\frac{1}{2N}\right), l=1,\cdots,L_N$}\ \text{and}\ \sum\limits_{l=1}^{L_N}\eta_l(\theta)=1, \ \text{for $\theta\in K^N$}.
\]
Now, choose some $\theta_l^*$ such that $\eta_l(\theta^*_l)>0$ and
set
\[
\Pi_{\theta}^N(t):=\sum\limits_{l=1}^{L_N}\Pi_{\theta_{l}^*}(t)\eta_l(\theta)I_{\{\theta\in K^N\}}.
\]
It follows from assumption (H1) that
\[
|\Pi_{\theta}(t)|\leq C(L)m_{\theta}(t)\left(1+|p^1_{\theta}(t)|+|q^1_{\theta}(t)|+|\overline{x}_{\theta}(t)|+|\overline{u}(t)|\right), \ \forall \theta\in\Theta,
\]
which together with inequality \eqref{myw907} and Lemma \ref{myw301} implies that
\[
\sup\limits_{\theta\in\Theta}\mathbb{E}\left[\int^T_0|\Pi_{\theta}(t)|dt\right]\leq C(L,T)\mathbb{E}\left[1+|x_0|^{2}+\int^T_0|\overline{u}(t)|^{2}dt\right]^{\frac{1}{2}}.
\]
Note that $\eta_l(\theta)=0$ whenever $\mu(\theta,\theta_l)\geq \frac{1}{2N}$. Therefore, we derive  that
\begin{align*}
&\mathbb{E}\left[\int^T_0\left|\Pi_{\theta}^N(t)-\Pi_{\theta}(t)\right|dt\right]\\
\leq &\sum\limits_{l=1}^{L_N}\mathbb{E}\left[\int^T_0\left|\Pi_{\theta_{l}^*}(t)-\Pi_{\theta}(t)\right|dt\right]\eta_l(\theta)I_{\{\theta\in K^N\}}+\mathbb{E}\left[\int^T_0|\Pi_{\theta}(t)|dt\right]I_{\{\theta\notin K^N\}}\\
\leq& \sup\limits_{\mu(\theta,\theta^{\prime})\leq \frac{1}{N}}\mathbb{E}\left[\int^T_0\left|\Pi_{\theta^{\prime}}(t)-\Pi_{\theta}(t)\right|dt\right]+C(L,T)\mathbb{E}\left[1+|x_0|^{2}+\int^T_0|\overline{u}(t)|^{2}dt\right]^{\frac{1}{2}}I_{\{\theta\notin K^N\}}.
\end{align*}
As a result, we get that
\[
\int_{\Theta}\mathbb{E}\left[\int^T_0\left|\Pi_{\theta}^N(t)-\Pi_{\theta}(t)\right|dt\right]\overline{Q}(d\theta)\leq  \sup\limits_{\mu(\theta,\theta^{\prime})\leq \frac{1}{N}}\mathbb{E}\left[\int^T_0\left|\Pi_{\theta^{\prime}}(t)-\Pi_{\theta}(t)\right|dt\right]+ \frac{C(L,T,x_0,\overline{u})}{N}.
\]
Recalling  Lemma \ref{myw2092} in appendix B, we have that
\[
\lim\limits_{N\rightarrow\infty}\int_{\Theta}\mathbb{E}\left[\int^T_0\left|\Pi_{\theta}^N(t)-\Pi_{\theta}(t)\right|dt\right]\overline{Q}(d\theta)=0.
\]
Since $\Pi_{\theta}^N(t)$ is a $\mathscr{F}$-progressively measurable process,  the desired result holds.
\end{proof}

Finally, applying Lemma \ref{myw20941} and Fubini's theorem to inequality \eqref{myw2094} yields that, for each $u\in\mathcal{U}[0,T]$,
\[
 \mathbb{E}\left[ \int^T_0\int_{\Theta}\left\langle m_{\theta}(t)\partial_uH_{\theta}(t,\overline{x}_{\theta}(t),\overline{y}_{\theta}(t),\overline{z}_{\theta}(t),\overline{u}(t),\overline{u}(t),p^1_{\theta}(t),q^1_{\theta}(t)), u(t)-\overline{u}(t)\right\rangle
\overline{Q}(d\theta)dt\right]\geq 0,
\]
which implies that, for any $u\in U$,
\begin{align}\label{myw2097}
\int_{\Theta}\left\langle m_{\theta}(t)\partial_uH_{\theta}(t,\overline{x}_{\theta}(t),\overline{y}_{\theta}(t),\overline{z}_{\theta}(t),\overline{u}(t),\overline{u}(t),p^1_{\theta}(t),q^1_{\theta}(t)), u-\overline{u}(t)\right\rangle
\overline{Q}(d\theta)\geq 0, \ \text{$dt\times d\mathbb{P}$-a.e.}
\end{align}

Summarizing the above analysis, we could get the main result of the section.
\begin{theorem}\label{myq511}
Suppose assumptions \emph{(H1)}-\emph{(H5)} are satisfied. Let $\overline{u}$ be an optimal control and $(\overline{x}_{\theta},\overline{y}_{\theta},\overline{z}_{\theta})$ be the corresponding trajectory. Then, there exists a reference probability $\overline{Q}\in\mathcal{Q}^{\overline{u}}$
and $(m_{\theta},p^1_{\theta},q^1_{\theta})$ satisfying equations \eqref{myw509}, \eqref{myw909}, such that the inequality \eqref{myw2097} holds.
\end{theorem}

\subsection{Sufficient condition}

In this section, we will discuss the sufficient condition for the optimal control.
For this purpose, denote
$h_{\theta}:[0,T]\times\mathbb{R}^n\times\mathbb{R}\times\mathbb{R}^d\times U\times\mathbb{R}^n\times\mathbb{R}\times\mathbb{R}^d\times U\times\mathbb{R}^n\times\mathbb{R}^{n\times d}\rightarrow \mathbb{R}$ by
\begin{align*}
h_{\theta}(t,x,y,z,u,x^{\prime},y^{\prime},z^{\prime},u^{\prime},p,q)=&\left\langle p, b_{\theta}(t,x,u)+\sum\limits_{i=1}^d\partial_{z^{\prime,i}}f_{\theta}(t,x^{\prime},y^{\prime},z^{\prime},u^{\prime})\sigma^i_{\theta}(t,x,u)\right\rangle\\
&\ +\sum\limits_{i=1}^d\left\langle q^{i}, \sigma^i_{\theta}(t,x,u)\right\rangle +f_{\theta}(t,x,y,z,u).
\end{align*}
It is obvious that  $H_{\theta}(t,x,y,z,u,u^{\prime},p,q)=h_{\theta}(t,x,y,z,u,x,y,z,u^{\prime},p,q)$.

\begin{theorem}\label{myq512}
Suppose conditions \emph{(H1)}-\emph{(H5)} hold. Assume that the function $h_{\theta}$ is convex with respect to $x,y,z,u$ and $\varphi_{\theta}$ is convex with respect to $x$.
Let $\overline{u}\in\mathcal{U}[0,T]$ and  $\overline{Q}\in\mathcal{Q}^{\overline{u}}$ satisfy that
\[
\int_{\Theta}\left\langle m_{\theta}(t)\partial_uH_{\theta}(t,\overline{x}_{\theta}(t),\overline{y}_{\theta}(t),\overline{z}_{\theta}(t),\overline{u}(t),\overline{u}(t),p^1_{\theta}(t),q^1_{\theta}(t)), u-\overline{u}(t)\right\rangle
\overline{Q}(d\theta)\geq 0, \ \forall u\in U, \ \text{$dt\times d\mathbb{P}$-a.e.,}
\]
where $(\overline{x}_{\theta},\overline{y}_{\theta},\overline{z}_{\theta})$ is the solution to equations \eqref{App1y} and \eqref{App1y11}  corresponding to the admissible control  $\overline{u}$,
and $(m_{\theta},p^1_{\theta},q^1_{\theta})$ satisfying adjoint equations \eqref{myw509}, \eqref{myw909}. Then, the admissible control $\overline{u}$ is an optimal control.
\end{theorem}

\begin{proof}
For each $u\in\mathcal{U}[0,T]$ and $\theta\in\Theta$, let $(x_{\theta},y_{\theta},z_{\theta})$ be the corresponding state
processes of equations \eqref{App1y}-\eqref{App1y11}. Denote  $(\alpha_{\theta},\beta_{\theta},\zeta_{\theta}):=(x_{\theta}-\overline{x}_{\theta},y_{\theta}-\overline{y}_{\theta},z_{\theta}-\overline{z}_{\theta})$. Then, it holds that
\begin{align*}
&\alpha_{\theta}(t)=\int^t_0\left(\partial_xb_{\theta}(s)\alpha_{\theta}(s)+A_{\theta}(s)\right)ds+\sum\limits_{i=1}^d\int^t_0\left(\partial_x\sigma^i_{\theta}(s)\alpha_{\theta}(s)+D^{i}_{\theta}(s)\right)dW^i(s),\\
&\beta_{\theta}(t)=J_{\theta}+\int^T_t\left(\partial_yf_{\theta}(s)\beta_{\theta}(s)+\partial_zf_{\theta}(s)(\zeta_{\theta}(s))^{\top}+L_{\theta}(s) \right)ds-\int^T_t\zeta_{\theta}(s)dW(s),
\end{align*}
where $J_{\theta}=\varphi_{\theta}(x_{\theta}(T))-\varphi_{\theta}(\overline{x}_{\theta}(T))$ and
\begin{align*}
&A_{\theta}(t)=b_{\theta}(t,x_{\theta}(t),u(t))-b_{\theta}(t)-\partial_xb_{\theta}(t)\alpha_{\theta}(t),\ D^{i}_{\theta}(t)=\sigma^i_{\theta}(t,x_{\theta}(t),u(t))-\sigma^i_{\theta}(t)-\partial_x\sigma^i_{\theta}(t)\alpha_{\theta}(t),\\
& L_{\theta}(t)=f_{\theta}(t,x_{\theta}(t),y_{\theta}(t),z_{\theta}(t),u(t))-f_{\theta}(t)
-\partial_yf_{\theta}(t)\beta_{\theta}(t)-\partial_zf_{\theta}(t)(\zeta_{\theta}(t))^{\top}.
\end{align*}
For convenience, set $h_{\theta}(t)=h_{\theta}(t,\overline{x}_{\theta}(t),\overline{y}_{\theta}(t),\overline{z}_{\theta}(t),\overline{u}(t),\overline{x}_{\theta}(t),\overline{y}_{\theta}(t),\overline{z}_{\theta}(t),\overline{u}(t),p^1_{\theta}(t),q^1_{\theta}(t))$.
Then, applying  It\^{o}'s formula to $\left\langle m_{\theta}(t)p_{\theta}^1(t), \alpha_{\theta}(t)\right\rangle- m_{\theta}(t)\beta_{\theta}(t)$ yields that
\begin{align*}
&\mathbb{E}\left[m_{\theta}(T)\partial_x\varphi_{\theta}(\overline{x}(T))\alpha_{\theta}(T)- m_{\theta}(T)J_{\theta}+\beta_{\theta}(0)\right]\\
&=\mathbb{E}\left[\int^T_0m_{\theta}(t)\left[-\partial_xf_{\theta}(t)\alpha_{\theta}(t)+\left\langle p_{\theta}^1(t),A_{\theta}(t)\right\rangle +\sum\limits_{i=1}^d \left\langle \partial_{z^i}f_{\theta}(t)p_{\theta}^1(t)+q^{1,i}_{\theta}(t),D^i_{\theta}(t)\right\rangle +L_{\theta}(t)\right]dt
\right]
\\
&=\mathbb{E}\left[\int^T_0m_{\theta}(t)\left[-\partial_xh_{\theta}(t)\alpha_{\theta}(t)-\partial_yh_{\theta}(t)\beta_{\theta}(t)-\partial_zh_{\theta}(t)(\zeta_{\theta}(t))^{\top}+h^{*}_{\theta}(t)-h_{\theta}(t)\right]dt
\right],
\end{align*}
where $h^*_{\theta}(t)=h_{\theta}(t,{x}_{\theta}(t),{y}_{\theta}(t),{z}_{\theta}(t),{u}(t),\overline{x}_{\theta}(t),\overline{y}_{\theta}(t),\overline{z}_{\theta}(t),\overline{u}(t),p^1_{\theta}(t),q^1_{\theta}(t))$.

Note that $h_{\theta}$ is convex with respect to $x,y,z,u$. Thus, we deduce that
\[
-\partial_xh_{\theta}(t)\alpha_{\theta}(t)-\partial_yh_{\theta}(t)\beta_{\theta}(t)-\partial_zh_{\theta}(t)(\zeta_{\theta}(t))^{\top}+h^{*}_{\theta}(t)-h_{\theta}(t)\geq \partial_uh_{\theta}(t)(u(t)-\overline{u}(t)).
\]
Similarly, we have that $\partial_x\varphi_{\theta}(\overline{x}(T))\alpha_{\theta}(T)\leq J_{\theta}.$
Note that \[\partial_uh_{\theta}(t)=\partial_uH_{\theta}(t,\overline{x}_{\theta}(t),\overline{y}_{\theta}(t),\overline{z}_{\theta}(t),\overline{u}(t),\overline{u}(t),p^1_{\theta}(t),q^1_{\theta}(t))).\]
Thus, it follows that
\begin{align*}
&\beta_{\theta}(0)\geq \mathbb{E}\left[\int^T_0m_{\theta}(t)\langle \partial_uh_{\theta}(t), u(t)-\overline{u}(t)\rangle dt\right],
\end{align*}
which together with inequality \eqref{myw2097} implies that \[
\int_{\Theta}{y}_{\theta}(0)\overline{Q}(d\theta)-\int_{\Theta}\overline{y}_{\theta}(0)\overline{Q}(d\theta)=\int_{\Theta}\beta_{\theta}(0)\overline{Q}(d\theta)\geq 0.\]
Consequently, in spirit of the fact that $\overline{Q}\in\mathcal{Q}^{\overline{u}}$, we could derive that
\[
J(u)-J(\overline{u})\geq \int_{\Theta}{y}_{\theta}(0)\overline{Q}(d\theta)-\int_{\Theta}\overline{y}_{\theta}(0)\overline{Q}(d\theta)\geq 0,
\]
which completes the proof.
\end{proof}

\section{A linear quadratic robust control problem}
For simplicity of presentation, suppose that $d=1$, i.e., the
Brownian motion is one-dimensional.

Assume that $\mathcal{U}[0,T]=\mathcal{M}^p(0,T;\mathbb{R}^k)$ for some constant $p>4$. Suppose that $\Theta=\{1,2\}$ is a discrete space and
\[
\mathcal{Q}=\{Q^{\lambda}:\ \lambda\in[0,1]\},
\]
where $Q^{\lambda}$ is the probability such that $Q^{\lambda}(\{1\})=\lambda$ and $Q^{\lambda}(\{2\})=1-\lambda$.

Consider the following linear quadratic control problem, where the state equation is given by
\begin{align}\label{myq986}
\begin{cases}
&x_{\theta}(t)=x_0+\int^t_0\left[A_{\theta}(s)x_{\theta}(s)+B_{\theta}(s)u(s)\right]ds+\int^t_0\left[C_{\theta}(s)x_{\theta}(s)+D_{\theta}(s)u(s)\right]dW(s),\\
&y_{\theta}(t)=\frac{1}{2}\langle G_{\theta}x_{\theta}(T),x_{\theta}(T)\rangle+\int^T_t\big\{E_{\theta}(s)y_{\theta}(s)+F(s)z_{\theta}(s)
  +\frac{1}{2}\big[\langle L_{\theta}(s)x_{\theta}(s),x_{\theta}(s)\rangle\\& \ \ \ \ \ \ \ \ \ \ \ \ \ \ \ \ \ \ \ \ \ \ \ \ \ \ \ \ \ \ +2\langle S_{\theta}(s)x_{\theta}(s),u(s)\rangle+\langle R_{\theta}(s)u(s),u(s)\rangle\big]\big\}ds -\int_{t}^{T}z_{\theta}(s)dW(s).
\end{cases}
\end{align}
Here, $G_{\theta}\in\mathbb{S}_n$ and $A_{\theta}, B_{\theta},C_{\theta},D_{\theta},E_{\theta}, F, L_{\theta},S_{\theta}$, $R_{\theta}$ are deterministic functions on $[0,T]$ satisfy the following conditions:
\begin{description}
\item[(H6)] $A_{\theta},C_{\theta}\in \mathcal{M}^{\infty}(0,T;\mathbb{R}^{n\times n})$, $B_{\theta}\in \mathcal{M}^{\infty}(0,T;\mathbb{R}^{n\times k})$, $D_{\theta}\in C(0,T;\mathbb{R}^{n\times k})$, $E_{\theta}, F\in \mathcal{M}^{\infty}(0,T;\mathbb{R})$, $L_{\theta}\in \mathcal{M}^{\infty}(0,T;\mathbb{S}_n)$, $S_{\theta}\in \mathcal{M}^{\infty}(0,T;\mathbb{R}^{k\times n})$,
 $R_{\theta}\in C(0,T;\mathbb{S}_k)$;
\item[(H7)]  $G_{\theta}\geq 0$ , $L_{\theta}(t)-S^{\top}_{\theta}(t)R^{-1}_{\theta}(t)S_{\theta}(t)\geq 0$ and $R_{\theta}(t)\gg 0$, i.e., there exists a constant $\delta>0$ such that
 $R_{\theta}(t)\geq \delta I_{k\times k}$ for each $t\in[0,T]$ and $\theta\in\{1,2\}$.
\end{description}
In this case, the cost function is given by
\[
J(u)=\sup\limits_{Q^{\lambda}\in\mathcal{Q}}\int_{\Theta}y_{\theta}(0)Q^{\lambda}(d\theta)=\max(y_{1}(0),y_2(0)).
\]

First, we characterize the explicit form of optimal control via our maximum principle.  Let $\overline{u}$ be an optimal control. Then, from Theorem \ref{myq511} and the definition of $(p^1_{\theta},q^1_{\theta})$, there exists a probability  ${Q}^{\overline{\lambda}}\in\mathcal{Q}$ such that $\max(\overline{y}_{1}(0),\overline{y}_2(0))=\overline{\lambda}\overline{y}_{1}(0)+(1-\overline{\lambda})\overline{y}_{2}(0)$ and
\begin{align} \label{myq50061}
\begin{split}
&\overline{\lambda}m_1(t)\left[(B^{\top}_{1}(t)+D^{\top}_1(t)F(t))p^1_1(t)+D^{\top}_1(t)q^1_1(t)+S_1(t)\overline{x}_1(t)+ R_1(t)\overline{u}(t)\right]\\
&+(1-\overline{\lambda})m_2(t)\left[(B^{\top}_{2}(t)+D^{\top}_2(t)F(t))p^1_2(t)+D^{\top}_2(t)q^1_2(t)+S_2(t)\overline{x}_2(t)+ R_2(t)\overline{u}(t)\right] =0,
\end{split}
\end{align}
where $(p^1_{\theta},q^1_{\theta})$ is the solution of the following adjoint equation:
\begin{align*}
\begin{cases}
&-dp^1_{\theta}(t)=\big\{{L}_{\theta}(t)\overline{x}_{\theta}(t)+S^{\top}_{\theta}(t)\overline{u}(t)+\big[{A}^{\top}_{\theta}(t)+E_{\theta}(t)I_{n\times n}+F(t){C}^{\top}_{\theta}(t)\big]p^1_{\theta}(t)\\
&\ \ \ \ \ \ \ \ \ \ \ \ \ \  \ \ \ \ \ \ \ +\big[F(t)I_{n\times n}+{C}^{\top}_{\theta}(t)\big]q^1_{\theta}(t)\big\}dt-q^1_{\theta}(t)dW(t), \\
&\ \ \ p^1_{\theta}(T)={G}_{\theta}\overline{x}_{\theta}(T).
\end{cases}
\end{align*}
Note that
\[
m_\theta(t)=\widetilde{m}_{\theta}(t)\exp\left(\int^t_0F(s)dW(s)-\frac{1}{2}\int^t_0 F^2(s)ds\right) \ \text{with}\ \widetilde{m}_{\theta}(t)=\exp\left(\int^t_0E_{\theta}(s)ds\right).
\]
It follows from equation \eqref{myq50061} that
\begin{align*}
\begin{split}
&\overline{\lambda}\widetilde{m}_1(t)\left[(B^{\top}_{1}(t)+D^{\top}_1(t)F(t))p^1_1(t)+D^{\top}_1(t)q^1_1(t)+S_1(t)\overline{x}_1(t)\right]\\
&+(1-\overline{\lambda})\widetilde{m}_2(t)\left[(B^{\top}_{2}(t)+D^{\top}_2(t)F(t))p^1_2(t)+D^{\top}_2(t)q^1_2(t)+S_2(t)\overline{x}_2(t)\right]\\
&+ \left[\overline{\lambda}\widetilde{m}_1(t)R_1(t)+(1-\overline{\lambda})\widetilde{m}_2(t)R_2(t)\right]\overline{u}(t) =0.
\end{split}
\end{align*}

For convenience, set
\[
\overline{x}=\left[
\begin{array}
[c]{cc}%
\overline{x}_1\\
\overline{x}_2
\end{array}
\right], \ B=\left[
\begin{array}
[c]{cc}%
B_1\\
B_2
\end{array}
\right],\ D=\left[
\begin{array}
[c]{cc}%
D_1\\
D_2
\end{array}
\right],\ p^1=\left[
\begin{array}
[c]{cc}%
\widetilde{m}_1p^1_1\\
\widetilde{m}_2p^1_2
\end{array}
\right],\ q^1=\left[
\begin{array}
[c]{cc}%
\widetilde{m}_1 q^1_1\\
\widetilde{m}_2 q^1_2
\end{array}
\right]
\]
and\[
\overline{\Lambda}=\left[
\begin{array}
[c]{cc}%
\overline{\lambda}I_{n\times n} & 0 \\
0 & (1-\overline{\lambda})I_{n\times n}
\end{array}
\right],\ R^{\overline{\lambda}}= \overline{\lambda}\widetilde{m}_1R_1+(1-\overline{\lambda})\widetilde{m}_2R_2,\  S=\left[
\begin{array}
[c]{cc}%
\widetilde{m}_1S_1 & \widetilde{m}_2S_2
\end{array}
\right].
\]
Thus, the optimal control $\overline{u}$ satisfies the following equation:
\begin{align} \label{myq5006}
\begin{cases}
&\max(\overline{y}_{1}(0),\overline{y}_2(0))=\overline{\lambda}\overline{y}_{1}(0)+(1-\overline{\lambda})\overline{y}_{2}(0), \\
&(B^{\top}(t)+D^{\top}(t)F)\overline{\Lambda}p^1(t)+D^{\top}(t)\overline{\Lambda}q^1(t)+S(t)\overline{\Lambda}\overline{x}(t)+R^{\overline{\lambda}}(t)\overline{u}(t)=0,
\end{cases}
\end{align}
where $(p^1,q^1)$ is the solution of the following adjoint equation:
\begin{align} \label{myq5007}
\begin{cases}
&-dp^1(t)=\big\{\widetilde{L}(t)\overline{x}(t)+S^{\top}(t)\overline{u}(t)+
\big[\widetilde{A}^{\top}+F(t)\widetilde{C}^{\top}(t)\big]p^1(t)\\
&\ \ \ \ \ \ \ \ \ \ \ \ \ \ \ \ \ \ \ +\big[F(t)I_{2n\times 2n}+\widetilde{C}^{\top}(t)\big]q^1(t)\big\}dt-q^1(t)dW(t), \\
&\ \ \ p^1(T)=\widetilde{G}\overline{x}(T),
\end{cases}
\end{align}
with
\[
\widetilde{A}=\left[
\begin{array}
[c]{cc}%
A_1& 0\\
0&A_2
\end{array}
\right],\ \widetilde{C}=\left[
\begin{array}
[c]{cc}%
C_1& 0\\
0& C_2
\end{array}
\right],\ \widetilde{L}=\left[
\begin{array}
[c]{cc}%
\widetilde{m}_1 L_1& 0\\
0& \widetilde{m}_2L_2
\end{array}
\right],\ \widetilde{G}=\left[
\begin{array}
[c]{cc}%
\widetilde{m}_1(T) G_1& 0\\
0& \widetilde{m}_2(T) G_2
\end{array}
\right].
\]

Next, we suppose that
\[
\overline{\Lambda}p^1(t)=P(t)\overline{x}(t)
\]
with $P\in C^1(0,T; \mathbb{S}_{2n})$.
Note that
\begin{align}\label{myq5017}
d\overline{x}(t)=[\widetilde{A}(t)\overline{x}(t)+B(t)\overline{u}(t)]dt+[\widetilde{C}(t)\overline{x}(t)+D(t)\overline{u}(t)]dW(t).
\end{align}
Then applying It\^{o}'s formula to $P(t)\overline{x}(t)$ and recalling equation \eqref{myq5007}, we derive that
\[
\overline{\Lambda}q^1(t)=P(t)\widetilde{C}(t)\overline{x}(t)+P(t)D(t)\overline{u}(t),
\]
and
\begin{align*}
\begin{split}
&[\dot{P}(t)+P(t)\widetilde{A}(t)+\widetilde{L}(t)\overline{\Lambda}]\overline{x}(t)+[P(t)B(t)+\overline{\Lambda}S^{\top}(t)]\overline{u}(t)\\
&\ \ \ \ +[\widetilde{A}^{\top}+F(t)\widetilde{C}^{\top}(t)]\overline{\Lambda}p^1(t)+[F(t)I_{2n\times 2n}+\widetilde{C}^{\top}(t)]\overline{\Lambda}q^1(t)=0.
\end{split}
\end{align*}
In the sequel, the variable $t$ will be suppressed for convenience.
Therefore, it follows from equation \eqref{myq5006} that the optimal control satisfies
\begin{align}\label{myq5010}
\overline{u}=-\left(R^{\overline{\lambda}}+D^{\top}PD\right)^{-1}\left((B+DF)^{\top}P+D^{\top}P\widetilde{C}+S\overline{\Lambda}\right)\overline{x},
\end{align}
where $P$ is the solution to the following Riccati equation on time interval $[0,T]$:
\begin{align} \label{myq5011}
\begin{cases}
&\dot{P}+P(\widetilde{A}+F\widetilde{C})+(\widetilde{A}+F\widetilde{C})^{\top}P+\widetilde{C}^{\top}P\widetilde{C}+\widetilde{L}\overline{\Lambda}\\
&-\left(P(B+DF)+\overline{\Lambda}S^{\top}+\widetilde{C}^{\top}PD\right)\left(R^{\overline{\lambda}}+D^{\top}PD\right)^{-1}\left((B+DF)^{\top}P+D^{\top}P\widetilde{C}+S\overline{\Lambda}\right)=0, \\
&P(T)=\overline{\Lambda}\widetilde{G}.
\end{cases}
\end{align}
\begin{remark}
	{\upshape
		In the  FBSDE \eqref{myq986}, we assume that   $F$ is independent of  parameter $\theta$ to ensure
		all the coefficients  are deterministic in the second equation of \eqref{myq5006}. Otherwise, the Riccati equation \eqref{myq5011} would be a
		BSDE instead of ordinary differential equation (ODE). For this topic, we refer the reader to  \cite{Ts1}.
	}
\end{remark}

To ensure the well-posedness of  the  Riccati equation \eqref{myq5011}, we need the following resut.
\begin{lemma} \label{myw6012}
Assume \emph{(H7)} hold. Then, we have	for each $\overline{\lambda}\in [0,1]$,
\[
\widetilde{L}\overline{\Lambda}-\overline{\Lambda}S^{\top}(R^{\overline{\lambda}})^{-1}S\overline{\Lambda}\geq 0.
\]
\end{lemma}
\begin{proof} For readers' convenience, we shall give the sketch of the proof. It suffices to prove that, for any $x,y\in\mathbb{R}^n$
\begin{align*}\begin{split}
&\left[
\begin{array}
[c]{cc}%
x^{\top} & y^{\top}
\end{array} \right]\widetilde{L}\overline{\Lambda}\left[
\begin{array}
[c]{cc}%
x\\
y
\end{array}
\right]-\left[
\begin{array}
[c]{cc}%
x^{\top} & y^{\top}
\end{array} \right]\overline{\Lambda}S^{\top}(R^{\overline{\lambda}})^{-1}S\overline{\Lambda}\left[
\begin{array}
[c]{cc}%
x\\
y
\end{array}
\right]\\
&=\overline{\lambda}\widetilde{m}_1x^{\top}L_1x+(1-\overline{\lambda})\widetilde{m}_2y^{\top}L_2y-\overline{\lambda}^2\widetilde{m}^2_1(S_1x)^{\top}(R^{\overline{\lambda}})^{-1}S_1x-(1-\overline{\lambda})^2\widetilde{m}^2_2(S_2y)^{\top}(R^{\overline{\lambda}})^{-1}S_2y\\
&\ \ \ \ -\overline{\lambda}(1-\overline{\lambda})\widetilde{m}_1\widetilde{m}_2\big[(S_1x)^{\top}(R^{\overline{\lambda}})^{-1}S_2y+(S_2y)^{\top}(R^{\overline{\lambda}})^{-1}S_1x\big]
\geq 0.
\end{split}
\end{align*}
 By the condition (H7), we have that
\[
\overline{\lambda}\widetilde{m}_1x^{\top}L_1x+(1-\overline{\lambda})\widetilde{m}_2y^{\top}L_2y\geq \overline{\lambda}\widetilde{m}_1(S_1x)^{\top}R_1^{-1}S_1x+(1-\overline{\lambda})\widetilde{m}_2(S_2y)^{\top}R_2^{-1}S_2y.
\]
Recalling the definition of $R^{\overline{\lambda}}$, we get that
\[
R^{-1}_1=[\overline{\lambda}\widetilde{m}_1I_{k\times k}+(1-\overline{\lambda})\widetilde{m}_2R^{-1}_1R_2](R^{\overline{\lambda}})^{-1} \ \text{and}\
R^{-1}_2=[\overline{\lambda}\widetilde{m}_1R^{-1}_2R_1+(1-\overline{\lambda})\widetilde{m}_2I_{k\times k}](R^{\overline{\lambda}})^{-1}.
\]
With the help of the above equations, we only need to prove that
  \begin{align*}\begin{split}  (S_1x)^{\top}R^{-1}_1R_2(R^{\overline{\lambda}})^{-1}S_1x+(S_2y)^{\top}R_2^{-1}R_1(R^{\overline{\lambda}})^{-1}S_2y  \geq (S_1x)^{\top}(R^{\overline{\lambda}})^{-1}S_2y+(S_2y)^{\top}(R^{\overline{\lambda}})^{-1}S_1x,
  \end{split}
  \end{align*}
 which is equivalent to
 \begin{align}\label{myw9001}
 \tilde{x}^{\top}R^{\overline{\lambda}}R^{-1}_1R_2\tilde{x}+\tilde{y}^{\top}R^{\overline{\lambda}}R^{-1}_2R_1\tilde{y}\geq \tilde{x}^{\top}R^{\overline{\lambda}}\tilde{y}+\tilde{y}^{\top}R^{\overline{\lambda}}\tilde{x},
 \end{align}
where $\tilde{x}=(R^{\overline{\lambda}})^{-1}S_1x$ and $\tilde{y}=(R^{\overline{\lambda}})^{-1}S_2y$. With the help of the fact that
$R^{\overline{\lambda}}= \overline{\lambda}\widetilde{m}_1R_1+(1-\overline{\lambda})\widetilde{m}_2R_2$,
the inequality \eqref{myw9001} reduces to
 \begin{align}\label{myw9002}
 \tilde{x}^{\top} R_2  \tilde{x}+\tilde{y}^{\top}R_1R^{-1}_2R_1\tilde{y}\geq \tilde{x}^{\top} R_1 \tilde{y}+ \tilde{y}^{\top} R_1 \tilde{x}\ \text{and} \
 \tilde{x}^{\top} R_2R^{-1}_1R_2  \tilde{x}+\tilde{y}^{\top}R_1\tilde{y}\geq \tilde{x}^{\top} R_2 \tilde{y}+ \tilde{y}^{\top} R_2 \tilde{x}.
 \end{align}

On the other hand, since $R_{\theta}$ is positive definite, there is a positive definite  matrix $R_{\theta}^{\frac{1}{2}}$ so that $R_{\theta}=R_{\theta}^{\frac{1}{2}}R_{\theta}^{\frac{1}{2}}$. It follows from
Cauchy-Schwartz inequality that
 \begin{align*}
 &\tilde{x}^{\top} R_2  \tilde{x}+\tilde{y}^{\top}R_1R^{-1}_2R_1\tilde{y}
 =(R_{2}^{\frac{1}{2}}\tilde{x})^{\top}(R_{2}^{\frac{1}{2}}\tilde{x})+(R_{2}^{-\frac{1}{2}}R_{1}\tilde{y})^{\top}(R_{2}^{-\frac{1}{2}}R_{1}\tilde{y})\geq \tilde{x}^{\top}R_{1}\tilde{y}+\tilde{y}^{\top}R_{1}\tilde{x},\\
 &\tilde{x}^{\top} R_2R^{-1}_1R_2  \tilde{x}+\tilde{y}^{\top}R_1\tilde{y}=(R_{1}^{-\frac{1}{2}}R_{2}\tilde{x})^{\top}(R_{1}^{-\frac{1}{2}}R_{2}\tilde{x})+
 (R_{1}^{\frac{1}{2}}\tilde{y})^{\top}(R_{1}^{\frac{1}{2}}\tilde{y})\geq \tilde{x}^{\top} R_2 \tilde{y}+ \tilde{y}^{\top} R_2 \tilde{x},
 \end{align*}
 which ends the proof.
\end{proof}

\begin{lemma}
Suppose that the conditions \emph{(H6)}-\emph{(H7)} hold.  Then, the Riccati equation \eqref{myq5011} admits a unique solution $P\geq 0.$
\end{lemma}
\begin{proof}
From the condition (H7) and Lemma \ref{myw6012}, it holds that
\[
\overline{\Lambda}\widetilde{G}\geq 0,\  \widetilde{L}\overline{\Lambda}-\overline{\Lambda}S^{\top}(R^{\overline{\lambda}})^{-1}S\overline{\Lambda}\geq 0.
\]
Thus, by Theorem 7.2 in Chap. 6 of \cite{YZ}, the  Riccati equation \eqref{myq5011} admits a unique solution \[P\in C^1(0,T;\mathbb{S}_{2n}).\] In particular, $P(t)\geq 0$ for each $t\in[0,T]$.
The proof is complete.
\end{proof}

Now, putting equation \eqref{myq5017} and equation \eqref{myq5010} together, we can get the explicit form of the optimal control $\overline{u}$ and the optimal state process $\overline{x}$, which  both depend on the constant $\overline{\lambda}$. Moreover, the  optimal robust cost is given by $\max(\overline{y}_1(0),\overline{y}_2(0))$, where $(\overline{x}_{\theta},\overline{y}_{\theta},\overline{u})$ satisfies equation \eqref{myq986}.

\begin{remark}{\upshape
Suppose that the cost function is given by $J(u)=y_1(0)$. Then, it is easy to check that the corresponding optimal control $\overline{u}_1=\overline{u}$ with $\overline{\lambda}=1$; see  \cite{YZ}.
}
\end{remark}

Finally, we  study the existence of the optimal control. By Theorem \ref{myq512}, if equation \eqref{myq5006} holds, then $\overline{u}$ is an optimal control. Thus, we need to discuss the existence of solution to equation \eqref{myq5006}.

\begin{theorem}\label{myq6011}
Suppose that the assumptions \emph{(H6)}-\emph{(H7)} hold. Then, there exist a constant $\overline{\lambda}^*\in [0,1]$ and  an admissible control $\overline{u}\in\mathcal{U}[0,T]$ satisfying equations \eqref{myq5006}-\eqref{myq5011}. Moreover,
$\overline{u}$ is the  optimal control.
\end{theorem}
\begin{proof}
Note that the optimal control $\overline{u}$ and the optimal  state process $(\overline{x}_{\theta},\overline{y}_{\theta})$
satisfies equations  \eqref{myq5017}, \eqref{myq5010},  \eqref{myq5011} and  \eqref{myq986} for some constant $\overline{\lambda}^*\in [0,1]$.
Then,
denote by $P^{\overline{\lambda}}$ the solution to the Riccati equation \eqref{myq5011} for each constant $\overline{\lambda}\in[0,1]$.
Similarly, we can  also define ${\overline{x}}^{\overline{\lambda}}_{\theta}$, ${\overline{u}}^{\overline{\lambda}}$ and ${\overline{y}}^{\overline{\lambda}}_{\theta}$. By the construction of the Riccati equation \eqref{myq5011}, the second equality of equation \eqref{myq5006} holds for any $({\overline{x}}^{\overline{\lambda}}_{\theta},{\overline{u}}^{\overline{\lambda}})$ with  $\overline{\lambda}\in[0,1]$.
Thus,  we only need to find a constant $\overline{\lambda}^*\in[0,1]$ so that the first equality of equation \eqref{myq5006} holds.
 The proof is divided into the following three steps.

{\bf Step 1 ($y^0_{1}(0)\leq y^0_{2}(0)$ or $y^1_{1}(0)\geq y^1_{2}(0)$).}
Set
\begin{align*}
(\overline{\lambda},\overline{u})=
\begin{cases}
& (0,\overline{u}^0), \ \  \text{if $y^0_{1}(0)\leq y^0_{2}(0)$}, \\
& (1,\overline{u}^1),\ \ \text{if $y^1_{1}(0)\geq y^1_{2}(0)$}.
\end{cases}
\end{align*}
Then, the above linear quadratic control problem with model uncertainty reduces to the classical case,
and one can easily check that the desired results hold.

{\bf Step 2 ($y^0_{1}(0)>y^0_{2}(0)$ and $y^1_{1}(0)< y^1_{2}(0)$).} Note that all coefficients are uniformly bounded. From the proof Theorem 7.2 in Chap. 6 of \cite{YZ}, we could get that $P^{\overline{\lambda}}$ is uniformly bounded on $[0,T]$. We claim that
\begin{align}\label{myq7001}
|P^{\overline{\lambda}}(t)-P^{\overline{\lambda}^{\prime}}(t)|\leq C(T,\delta,\ell^*)|\overline{\lambda}-\overline{\lambda}^{\prime}|, \ \forall t\in[0,T],
\end{align}
whose proof will be given in step 3 and $\ell^*=(\widetilde{A},\widetilde{C},\widetilde{L},\widetilde{G},B,S,D,E_1,E_2,F)$.
Applying Lemma \ref{myq1} in appendix A and recalling equations \eqref{myq5017}, \eqref{myq5010}, we obtain that
\[
\mathbb{E}\left[\sup\limits_{0\leq t\leq T}\left|\overline{x}^{\overline{\lambda}}(t)-\overline{x}^{\overline{\lambda}^{\prime}}(t)\right|^4+\sup\limits_{0\leq t\leq T}\left|\overline{u}^{\overline{\lambda}}(t)-\overline{u}^{\overline{\lambda}^{\prime}}(t)\right|^4\right]\leq C(T,\delta,\ell^*)|\overline{\lambda}-\overline{\lambda}^{\prime}|^4,
\]
which together with Lemma \ref{myq2} in appendix A implies that
\[
\left|\overline{y}^{\overline{\lambda}}_1(0)-\overline{y}^{\overline{\lambda}^{\prime}}_1(0)\right|+
\left|\overline{y}^{\overline{\lambda}}_2(0)-\overline{y}^{\overline{\lambda}^{\prime}}_2(0)\right|\leq  C(T,\delta,\ell^*)|\overline{\lambda}-\overline{\lambda}^{\prime}|.
\]
It follows that $\overline{y}^{\overline{\lambda}}_1(0)-\overline{y}^{\overline{\lambda}}_2(0)$ is continuous in $\overline{\lambda}$.

Note that $y^0_{1}(0)-y^0_{2}(0)>0$ and  $y^1_{1}(0)- y^1_{2}(0)<0$.  Therefore, by intermediate value theorem, there exists a constant $\overline{\lambda}^*\in (0,1)$ such that
\[
\overline{y}^{\overline{\lambda}^*}_1(0)=\overline{y}^{\overline{\lambda}^*}_2(0).
\]
Moreover, it holds that
\[
\max\left(\overline{y}^{\overline{\lambda}^*}_1(0),\overline{y}^{\overline{\lambda}^*}_2(0)\right)=\overline{\lambda}^*\overline{y}^{\overline{\lambda}^*}_1(0)+(1-\overline{\lambda}^*)\overline{y}^{\overline{\lambda}^*}_2(0),
\]
which is the desired result.

{\bf Step 3 (The proof of inequality \eqref{myq7001}).}
Denote $\widehat{P}=P^{\overline{\lambda}}-P^{\overline{\lambda}^{\prime}}$. Then $\widehat{P}$ satisfies the following linear ODE:
\begin{align*}
\begin{cases}
&\dot{\widehat{P}}+\widehat{P}(\widetilde{A}+F\widetilde{C})+(\widetilde{A}+F\widetilde{C})^{\top}\widehat{P}+\widetilde{C}^{\top}\widehat{P}\widetilde{C}+\widetilde{L}(\overline{\Lambda}-\overline{\Lambda}^{\prime})\\&
-\left(\widehat{P}(B+DF)+(\overline{\Lambda}-\overline{\Lambda}^{\prime})S^{\top}+\widetilde{C}^{\top}\widehat{P}D\right)\widehat{R}^{-1}\left((B+DF)^{\top}P^{\overline{\lambda}}+D^{\top}P^{\overline{\lambda}}\widetilde{C}+S\overline{\Lambda}\right)\\
&+\left(P^{\overline{\lambda}^{\prime}}(B+DF)+\overline{\Lambda}^{\prime}S^{\top}+\widetilde{C}^{\top}P^{\overline{\lambda}^{\prime}}D\right)\widehat{R}^{-1}(\check{R}+D^{\top}\widehat{P}D)(\widehat{R}^{\prime})^{-1}\left((B+DF)^{\top}P^{\overline{\lambda}}+D^{\top}P^{\overline{\lambda}}\widetilde{C}+S\overline{\Lambda}\right)\\
&-\left(P^{\overline{\lambda}^{\prime}}(B+DF)+\overline{\Lambda}^{\prime}S^{\top}+\widetilde{C}^{\top}P^{\overline{\lambda}^{\prime}}D\right)(\widehat{R}^{\prime})^{-1}\left((B+DF)^{\top}\widehat{P}+D^{\top}\widehat{P}\widetilde{C}+S(\overline{\Lambda}-\overline{\Lambda}^{\prime})\right)=0, \\
&\widehat{P}(T)=(\overline{\Lambda}-\overline{\Lambda}^{\prime})\widetilde{G},
\end{cases}
\end{align*}
where  $\check{R}=R^{\overline{\lambda}}-R^{\overline{\lambda}^{\prime}},\widehat{R}=(R^{\overline{\lambda}}+D^{\top}P^{\overline{\lambda}}D)$ and $\widehat{R}^{\prime}=(R^{\overline{\lambda}^{\prime}}+D^{\top}P^{\overline{\lambda}^{\prime}}D)$.
Note that $\widehat{R}^{-1}$ and $(\widehat{R}^{\prime})^{-1}$ are uniformly bounded due to the assumption (H7). Then, using Gronwall's
inequality, we could deduce that
\[
|\widehat{P}(t)|\leq C(T,\delta,\ell^*)|\overline{\lambda}-\overline{\lambda}^{\prime}|, \ \forall t\in[0,T],
\]
which completes the proof.
\end{proof}

\section*{Acknowledgments}

 The authors would like to thank the editor and two anonymous referees for their careful reading and helpful suggestions,  which have greatly improved the presentation.

\appendix
\renewcommand\thesection{\normalsize Appendix A:  SDEs and BSDEs}
\section{ }

\renewcommand\thesection{A}

In this appendix, we state some well-known results about  SDEs and BSDEs for readers' convenience.
First,  consider the
following forward SDEs on $[0,T]$:
\begin{align} \label{App1y1}
&x(t)=x_0+\int^t_0b(s,x(s))ds+\int^t_0\sigma(s,x(s))dW(s),
\end{align}
where $b:[0,T]\times\Omega\times\mathbb{R}^{n}\rightarrow
\mathbb{R}^{n}$ and $\sigma:[0,T]\times\Omega\times\mathbb{R}^{n}\rightarrow
\mathbb{R}^{n\times d} $ satisfy the following assumptions:
\begin{description}
\item[(B1)] For each $x\in\mathbb{R}^n$, $b(\cdot,x)\in \mathcal{H}^{1,p}(0,T;\mathbb{R}^n)$ and $\sigma(\cdot,x)\in \mathcal{H}^p(0,T;\mathbb{R}^{n\times d})$ for some $p\geq 2$;
\item[(B2)] There exists some positive constant $L$  such that for any $x,x^{\prime}\in\mathbb{R}^n$,%
\begin{align*}
&|b(t,x)-b(t,x^{\prime})|+|\sigma(t,x)-\sigma(t,x^{\prime})|\leq
L|x-x^{\prime}|.
\end{align*}
\end{description}

\begin{lemma}\label{myq1}
Assume  that the conditions \emph{(B1)} and \emph{(B2)} hold. Then, the SDE \eqref{App1y1} admits a unique solution $x\in \mathcal{S}^p(0,T;\mathbb{R}^n).$ Moreover, it holds that
\[
\mathbb{E}\left[\sup_{0\leq t\leq T}|x(t)|^p\right]\leq  C(L,T,p) \mathbb{E}\left[|x_0|^p+\left(\int^T_0|b(t,0)|dt\right)^p+ \left(\int^T_0|\sigma(t,0)|^2dt\right)^{\frac{p}{2}} \right].
\]
\end{lemma}

For the proof of  Lemma  \ref{myq1}, we refer the reader to \cite{Zhang}.
Next, consider the
following backward SDEs on $[0,T]$:
\begin{align} \label{App1y2}
&y(t)=\xi+\int^T_tf(s,y(s),z(s))ds-\int_{t}^{T}z(s)dW(s),
\end{align}
where  $\xi\in L^p(\mathscr{F}_T;\mathbb{R}^m)$, $p>1$ and $f:[0,T]\times\Omega\times\mathbb{R}^m\times\mathbb{R}^{m\times d}\rightarrow
\mathbb{R}^m$ satisfies the following assumptions:
\begin{description}
\item[(B3)] For each $(y,z)\in\mathbb{R}\times\mathbb{R}^{d}$, $f(\cdot,y,z)\in \mathcal{H}^{1,p}(0,T;\mathbb{R}^m)$;
\item[(B4)] There exists some positive constant $L$  such that for any $(y,z),(y^{\prime},z^{\prime})\in\mathbb{R}^m\times\mathbb{R}^{m\times d}$,%
\begin{align*}
&|f(t,y,z)-f(t,y^{\prime},z^{\prime})|\leq
L(|y-y^{\prime}|+|z-z^{\prime}|).
\end{align*}
\end{description}

\begin{lemma}[\cite{BD}]\label{myq2}
Assume  that the conditions \emph{(B3)} and \emph{(B4)} hold. Then, the BSDE \eqref{App1y2} admits a unique solution $(y,z)\in \mathcal{S}^p(0,T;\mathbb{R}^m)\times  \mathcal{H}^p(0,T;\mathbb{R}^{m\times d}).$ Moreover, it holds that
\[
\mathbb{E}\left[\sup_{0\leq t\leq T}|y(t)|^p+\left(\int^T_0|z(t)|^2dt\right)^{\frac{p}{2}}\right]\leq  C(L,T,p) \mathbb{E}\left[|\xi|^p+\left(\int^T_0|f(t,0,0)|dt\right)^p \right].
\]
\end{lemma}

\appendix
\renewcommand\thesection{\normalsize Appendix B: The complement proofs}
\section{ }
\renewcommand\thesection{B}
\begin{lemma} \label{myw209} Assume that \emph{(H1)}-\emph{(H4)}  hold. Then,
\[
\lim\limits_{\epsilon\rightarrow 0}\sup\limits_{\mu(\theta,\theta^{\prime})\leq \epsilon}\mathbb{E}\left[\sup_{0\leq t\leq T}|x_{\theta}(t)-x_{\theta^{\prime}}(t)|^4+\sup_{0\leq t\leq T}|y_{\theta}(t)-y_{\theta^{\prime}}(t)|^2+\int^T_0|z_{\theta}(t)-z_{\theta^{\prime}}(t)|^2dt\right]=0.
\]
\end{lemma}
\begin{proof}
For any $\theta,\theta^{\prime}\in\Theta$, set $(\alpha,\beta,\zeta):=(x_{\theta}-x_{\theta^{\prime}},y_{\theta}-y_{\theta^{\prime}},z_{\theta}-z_{\theta^{\prime}})$.
The proof is divided into the following two steps. For convenience, we omit the argument $u$.

{\bf Step 1 ($x$-estimate).} Denote
\begin{align*}
 &A(t)=\int^1_0 \partial_x b_{\theta^{\prime}}(t,{x}_{\theta^{\prime}}(t)+\lambda({x}_{\theta}(t)-{x}_{\theta^{\prime}}(t)))d\lambda,\
 B^{i}(t)=\int^1_0 \partial_x \sigma^i_{\theta^{\prime}}(t,{x}_{\theta^{\prime}}(t)+\lambda({x}_{\theta}(t)-{x}_{\theta^{\prime}}(t)))d\lambda,\\
 &C(t)= b_{\theta}(t,{x}_{\theta}(t))-b_{\theta^{\prime}}(t,{x}_{\theta}(t)), \ D^i(t)= \sigma^i_{\theta}(t,{x}_{\theta}(t))-\sigma^i_{\theta^{\prime}}(t,{x}_{\theta}(t)).
\end{align*}
Thus, the process $\alpha$  satisfies the following SDE:
\begin{align*}
\alpha(t)=\int^t_0\left(A(s)\alpha(s)+C(s)\right)ds+\sum\limits_{i=1}^d\int^t_0\left(B^i(s)\alpha(s)+D^i(s)\right)dW^i(s).
\end{align*}
Applying Lemma \ref{myq1}   yields that
\begin{align*}
\mathbb{E}\left[\sup_{0\leq t\leq T}|\alpha(t)|^4\right]\leq C(L,T)\mathbb{E}\left[\int^T_0\left(|C(s)|^4+\sum\limits_{i=1}^d|D^i(s)|^4\right)ds  \right].
\end{align*}
From assumption (H4), we have that for each $N>0$,
\[
|\ell(t)|\leq \overline{\omega}_N(\mu(\theta,\theta^{\prime}))+C(L)(1+|x_{\theta}(t)|+|u(t)|)(I_{\{|x_{\theta}(t)|\geq N\}}+I_{\{|u(t)|\geq N\}}),
\]
where $\ell=C,D^i$.
Then, by a similar analysis as Lemma \ref{myw201}, we could  get that
\begin{align*}
\mathbb{E}\left[\sup_{0\leq t\leq T}|x_{\theta}(t)-x_{\theta^{\prime}}(t)|^4\right]\leq C(L,T,x_0,u,p)\left( |\overline{\omega}_N(\mu(\theta,\theta^{\prime}))|^4+N^{\frac{4-p}{p}}\right), \ \forall N>0,
\end{align*}
which implies that
\begin{align}\label{myw607}
\lim\limits_{\epsilon\rightarrow 0}\sup\limits_{\mu(\theta,\theta^{\prime})\leq \epsilon}\mathbb{E}\left[\sup_{0\leq t\leq T}|x_{\theta}(t)-x_{\theta^{\prime}}(t)|^4\right]=0.
\end{align}

{\bf Step 2 ($y$-estimate).}  Denote $J^1=\varphi_{\theta^{\prime}}(x_{\theta}(T))
-\varphi_{\theta^{\prime}}(x_{\theta^{\prime}}(T))$,
$J^2=\varphi_{\theta}(x_{\theta}(T))
-\varphi_{\theta^{\prime}}(x_{\theta}(T))$   and $\gamma=(x,y,z)$. Set
\begin{align*}
&E(t)=\int^1_0 \partial_x f_{\theta^{\prime}}(t,\gamma_{\theta^{\prime}}(t)+\lambda(\gamma_{\theta}(t)-\gamma_{\theta^{\prime}}(t)))d\lambda,
\ F(t)=\int^1_0 \partial_y f_{\theta^{\prime}}(t,\gamma_{\theta^{\prime}}(t)+\lambda(\gamma_{\theta}(t)-\gamma_{\theta^{\prime}}(t)))d\lambda,\\
 &G(t)= \int^1_0 \partial_z f_{\theta^{\prime}}(t,\gamma_{\theta^{\prime}}(t)+\lambda(\gamma_{\theta}(t)-\gamma_{\theta^{\prime}}(t)))d\lambda,\  H(t)=f_{\theta}(t,\gamma_{\theta}(t))-f_{\theta^{\prime}}(t,\gamma_{\theta}(t)).
\end{align*}
Then, it follows from Lemma \ref{myq2}   that
\begin{align*}
\mathbb{E}\left[\sup_{0\leq t\leq T}|\beta(t)|^2+\int^T_0|\zeta(t)|^2dt\right]\leq C(L,T)\mathbb{E}\left[|J^1|^2+|J^2|^2+\left|\int^T_0|E(t)\alpha(t)+H(t)|dt\right|^2\right],
\end{align*}
According to assumption (H1) and  equation  \eqref{myw607}, we get that
\[
\lim\limits_{\epsilon\rightarrow 0}\mathbb{E}\left[|J^1|^2+\left|\int^T_0|E(t)\alpha(t)|dt\right|^2\right]
\leq C(L,T,x_0,u)\lim\limits_{\epsilon\rightarrow 0}\sup\limits_{\mu(\theta,\theta^{\prime})\leq \epsilon}\mathbb{E}\left[\sup_{0\leq t\leq T}|x_{\theta}(t)-x_{\theta^{\prime}}(t)|^4\right]^{\frac{1}{2}}=0.
\]
On the other hand, by a similar analysis as in step 1, we have that
\[
\lim\limits_{\epsilon\rightarrow 0}\mathbb{E}\left[|J^2|^2+\left|\int^T_0|H(t)|dt\right|^2\right]=0,
\]
which completes the proof.
\end{proof}

\begin{lemma} \label{myw2091} Suppose that \emph{(H1)}-\emph{(H4)}  hold. Then,
\[
\lim\limits_{\epsilon\rightarrow 0}\sup\limits_{\mu(\theta,\theta^{\prime})\leq \epsilon}\mathbb{E}\left[\sup_{0\leq t\leq T}|\widehat{x}_{\theta}(t)-\widehat{x}_{\theta^{\prime}}(t)|^4+\sup_{0\leq t\leq T}|\widehat{y}_{\theta}(t)-\widehat{y}_{\theta^{\prime}}(t)|^2+\int^T_0|\widehat{z}_{\theta}(t)-\widehat{z}_{\theta^{\prime}}(t)|^2dt\right]=0.
\]
\end{lemma}

\begin{proof}
For any $\theta,\theta^{\prime}\in\Theta$, set $(\widehat{\alpha},\widehat{\beta},\widehat{\zeta}):=(\widehat{x}_{\theta}-\widehat{x}_{\theta^{\prime}},\widehat{y}_{\theta}-\widehat{y}_{\theta^{\prime}},\widehat{z}_{\theta}-\widehat{z}_{\theta^{\prime}})$.
The proof is divided into the following two steps. For convenience, we omit the argument $\overline{u}$.

{\bf Step 1 ($\widehat{x}$-estimate).} Denote
\begin{align*}
 &\widehat{C}(t)=(\partial_xb_{\theta}(s)-\partial_xb_{\theta^{\prime}}(s))\widehat{x}_{\theta}(s)+(\partial_ub_{\theta}(s)-\partial_ub_{\theta^{\prime}}(s))(u(s)-\overline{u}(s)),\\
& \widehat{D}^i(t)=(\partial_x\sigma^i_{\theta}(s)-\partial_x\sigma^i_{\theta^{\prime}}(s))\widehat{x}_{\theta}(s)+(\partial_u\sigma^i_{\theta}(s)-\partial_u\sigma^i_{\theta^{\prime}}(s))(u(s)-\overline{u}(s)).
\end{align*}
Then, applying Lemma \ref{myq1}  in appendix  yields that
\begin{align*}
\mathbb{E}\left[\sup_{0\leq t\leq T}|\widehat{\alpha}(t)|^4\right]\leq C(L,T)\mathbb{E}\left[\int^T_0\left(|\widehat{C}(s)|^4+\sum\limits_{i=1}^d|\widehat{D}^i(s)|^4\right)ds  \right].
\end{align*}

By the definition, it holds that
\[
\partial_xb_{\theta}(s)-\partial_xb_{\theta^{\prime}}(s)=\partial_xb_{\theta}(s,\overline{x}_{\theta}(s))-
\partial_xb_{\theta^{\prime}}(s,\overline{x}_{\theta}(s))+\partial_xb_{\theta^{\prime}}(s,\overline{x}_{\theta}(s))-\partial_xb_{\theta^{\prime}}(s,\overline{x}_{\theta^{\prime}}(s)).
\]
By a similar analysis as Lemma \ref{myw209} above, we could  get that
\begin{align*}
\lim\limits_{\epsilon\rightarrow 0}\sup\limits_{\mu(\theta,\theta^{\prime})\leq \epsilon}\mathbb{E}\left[\int^T_0\left|\partial_xb_{\theta}(s,\overline{x}_{\theta}(s))-
\partial_xb_{\theta^{\prime}}(s,\overline{x}_{\theta}(s))\right|^4|\widehat{x}_{\theta}(s)|^4ds  \right]=0.
\end{align*}
By assumption (H3), for each $\varepsilon>0$, there exists a $\delta>0$ such that
\[
|l(s,x,u)-l(s,x^{\prime},u)|\leq \varepsilon \ \text{for $l=\partial_xb_{\theta^{\prime}},\partial_ub_{\theta^{\prime}},\partial_x\sigma_{\theta^{\prime}},\partial_u\sigma_{\theta^{\prime}}$},
\]
whenever $|x-x^{\prime}|\leq \delta$.
From the above  inequality, we get that
\begin{align*}
|\partial_xb_{\theta^{\prime}}(s,\overline{x}_{\theta}(s))-
\partial_xb_{\theta^{\prime}}(s)|\leq& \varepsilon+C(L)I_{\{|\overline{x}_{\theta}(s)-\overline{x}_{\theta^{\prime}}(s)|\geq \delta\}}.
\end{align*}
Thus, a direct computation yields that
\begin{align*}
&\mathbb{E}\left[\int^T_0\left|\partial_xb_{\theta^{\prime}}(s,\overline{x}_{\theta}(s))-
\partial_xb_{\theta^{\prime}}(s)\right|^4|\widehat{x}_{\theta}(s)|^4ds  \right]
\\
&\leq
C(L,T,x_0,\overline{u})\left(|\varepsilon|^4+\delta^{\frac{4-p}{p}}{\mathbb{E}\left[\sup\limits_{0\leq t\leq T}|\overline{x}_{\theta}(t)-\overline{x}_{\theta^{\prime}}(t)|\right]^{\frac{p-4}{p}}}\right),
\end{align*}
which together with  Lemma \ref{myw209} implies that
\begin{align*}
\lim\limits_{\epsilon\rightarrow 0}\sup\limits_{\mu(\theta,\theta^{\prime})\leq \epsilon}\mathbb{E}\left[\int^T_0\left|\partial_xb_{\theta^{\prime}}(s,\overline{x}_{\theta}(s))-
\partial_xb_{\theta^{\prime}}(s)\right|^4|\widehat{x}_{\theta}(s)|^4ds  \right]
\leq
C(L,T,x_0,\overline{u})|\varepsilon|^4.
\end{align*}
Sending $\varepsilon\rightarrow 0$, we obtain that the left side  is equal to $0$.
As a result, we deduce that
\begin{align*}
\lim\limits_{\epsilon\rightarrow 0}\sup\limits_{\mu(\theta,\theta^{\prime})\leq \epsilon}\mathbb{E}\left[\int^T_0\left|\partial_xb_{\theta}(s)-\partial_xb_{\theta^{\prime}}(s)\right|^4|\widehat{x}_{\theta}(s)|^4ds  \right]=0.
\end{align*}
Using the same method, we could  derive that
\[
\lim\limits_{\epsilon\rightarrow 0}\sup\limits_{\mu(\theta,\theta^{\prime})\leq \epsilon}\mathbb{E}\left[\sup_{0\leq t\leq T}|\widehat{x}_{\theta}(t)-\widehat{x}_{\theta^{\prime}}(t)|^4\right]=0.
\]
{\bf Step 2 ($\widehat{y}$-estimate).}
 Set $\gamma_{\theta}=(x_{\theta},y_{\theta},z_{\theta})$ and
\begin{align*}
& \widehat{J}^1=\partial_x\varphi_{\theta^{\prime}}(\overline{x}_{\theta^{\prime}}(T))\widehat{\alpha}(T),\ \widehat{J}^2=(\partial_x\varphi_{\theta}(\overline{x}_{\theta}(T))-\partial_x\varphi_{\theta^{\prime}}(\overline{x}_{\theta^{\prime}}(T)))\widehat{x}_{\theta}(T),\
\widehat{E}(t)=\partial_x f_{\theta^{\prime}}(t,\overline{\gamma}_{\theta^{\prime}}(t))
\\
& \widehat{F}(t)=\partial_y f_{\theta^{\prime}}(t,\overline{\gamma}_{\theta^{\prime}}(t)),\ \widehat{G}(t)=\partial_z f_{\theta^{\prime}}(t,\overline{\gamma}_{\theta^{\prime}}(t)),\widehat{H}(t)=[\partial_u f_{\theta}(t,\overline{\gamma}_{\theta}(t))-\partial_u f_{\theta^{\prime}}(t,\overline{\gamma}_{\theta^{\prime}}(t))](u(t)-\overline{u}(t)) \\& +[\partial_x f_{\theta}(t,\overline{\gamma}_{\theta}(t))-\widehat{E}(t)]\widehat{x}_{\theta}(t)
 +[\partial_y f_{\theta}(t,\overline{\gamma}_{\theta}(t))-\widehat{F}(t)]\widehat{y}_{\theta}(t)  +[\partial_z f_{\theta}(t,\overline{\gamma}_{\theta}(t))-\widehat{G}(t)](\widehat{z}_{\theta}(t))^{\top}.
\end{align*}
Thus, applying Lemma \ref{myq2}  in appendix  A yields that
\begin{align*}
\mathbb{E}\left[\sup_{0\leq t\leq T}|\widehat{\beta}(t)|^2+\int^T_0|\widehat{\zeta}(t)|^2dt\right]\leq C(L,T)\mathbb{E}\left[\left|\widehat{J}^1\right|^2+\left|\widehat{J}^2\right|^2+\left|\int^T_0\left(|\widehat{E}(t)\widehat{\alpha}(t)|+\left|\widehat{H}(t)\right|\right)dt\right|^2\right].
\end{align*}
According to assumption (H1), we have that
\[
\lim\limits_{\epsilon\rightarrow 0}\sup\limits_{\mu(\theta,\theta^{\prime})\leq \epsilon}\mathbb{E}\left[\left|\widehat{J}^1\right|^2+\left|\int^T_0|\widehat{E}(t)\widehat{\alpha}(t)|dt\right|^2\right]\leq C(L,T,x_0,\overline{u})\lim\limits_{\epsilon\rightarrow 0}\sup\limits_{\mu(\theta,\theta^{\prime})\leq \epsilon}\mathbb{E}\left[\sup_{0\leq t\leq T}|\widehat{\alpha}(t)|^4\right]^{\frac{1}{2}}=0.
\]
On the other hand, by a similar analysis as in step 1, we could get that
\[
\lim\limits_{\epsilon\rightarrow 0}\sup\limits_{\mu(\theta,\theta^{\prime})\leq \epsilon}\mathbb{E}\left[\left|\widehat{J}^2\right|^2+\left|\int^T_0\left|\widehat{H}(t)\right|dt\right|^2\right]=0,
\]
which ends the proof.
\end{proof}

\begin{lemma} \label{myw2092} Suppose that \emph{(H1)}-\emph{(H4)}  hold. Then,
\[
\lim\limits_{\epsilon\rightarrow 0}\sup\limits_{\mu(\theta,\theta^{\prime})\leq \epsilon} \mathbb{E}\left[ \int^T_0\left|\Pi_{\theta}(t)
-\Pi_{\theta^{\prime}}(t)\right| dt
\right]=0,
\]
where the process $\Pi_{\theta}$ is given by equation \eqref{myw2093}.
\end{lemma}
\begin{proof} The main idea is from Lemma \ref{myw2091} and we only give the sketch of the proof. For convenience, we omit the argument $\overline{u}$.
Using the same method as  Lemma \ref{myw2091}, we  derive  that for each $q>2$ \begin{align}\label{myw901}
\lim\limits_{\epsilon\rightarrow 0}\sup\limits_{\mu(\theta,\theta^{\prime})\leq \epsilon}\mathbb{E}\left[\sup_{0\leq t\leq T}|{m}_{\theta}(t)-m_{\theta^{\prime}}(t)|^q+\sup_{0\leq t\leq T}|{p}^1_{\theta}(t)-p^1_{\theta^{\prime}}(t)|^4+\left|\int^T_0|{q}^1_{\theta}(t)-q^1_{\theta^{\prime}}(t)|^2dt\right|^2\right]=0.
\end{align}
Then, by a similar analysis as step 1 in the proof of Lemma \ref{myw2091}, we  could get that\begin{align*}
&\lim\limits_{\epsilon\rightarrow 0}\sup\limits_{\mu(\theta,\theta^{\prime})\leq \epsilon}\mathbb{E}\left[\int^T_0\left(\left|\ell_{\theta}(t)p^1_{\theta}
(t)-\ell_{\theta^{\prime}}(t)p^1_{\theta^{\prime}}(t)\right|^4+\left|\partial_uf_{\theta}(t,\overline{\gamma}_{\theta}(t))-\partial_uf_{\theta^{\prime}}(t,\overline{\gamma}_{\theta^{\prime}}(t))\right|^4\right)dt\right]=0
\end{align*}
where $\overline{\gamma}_{\theta}=(\overline{x}_{\theta},\overline{y}_{\theta},\overline{z}_{\theta})$ and $\ell_{\theta}(t)$ is  $(\partial_ub_{\theta}(t,\overline{x}_{\theta}(t)))^{\top}$, $\partial_{z^i}f_{\theta}(t,\overline{\gamma}_{\theta}(t))(\partial_u\sigma^i_{\theta}(t,\overline{x}_{\theta}(t)))^{\top}$, which together with equations \eqref{myw907} and \eqref{myw901}  implies that
\begin{align*}
&\lim\limits_{\epsilon\rightarrow 0}\sup\limits_{\mu(\theta,\theta^{\prime})\leq \epsilon}\mathbb{E}\bigg[\int^T_0\big(\left|m_{\theta}(t)\ell_{\theta}(t)p^1_{\theta}
(t)-m_{\theta^{\prime}}(t)\ell_{\theta^{\prime}}(t)p^1_{\theta^{\prime}}(t)\right|\\
&\ \ \ \ \ \ \ \ \ \ \ \ \ \ \ \ \ \ \ \ \ \ \ \ \ \ \ \ +\left|m_{\theta}(t)\partial_uf_{\theta}(t,\overline{\gamma}_{\theta}(t))-m_{\theta^{\prime}}(t)\partial_uf_{\theta^{\prime}}(t,\overline{\gamma}_{\theta^{\prime}}(t))\right|\big)dt\bigg]=0.
\end{align*}
On the other hand, a similar analysis yields that
\[
\lim\limits_{\epsilon\rightarrow 0}\sup\limits_{\mu(\theta,\theta^{\prime})\leq \epsilon}\mathbb{E}\left[\int^T_0\left|m_{\theta}(t)(\partial_u\sigma^i_{\theta}(t,\overline{x}_{\theta}(t)))^{\top}q^{1,i}_{\theta}
(t)-m_{\theta^{\prime}}(t)(\partial_u\sigma^i_{\theta^{\prime}}(t,\overline{x}_{\theta^{\prime}}(t)))^{\top}q^{1,i}_{\theta^{\prime}}
(t)\right|dt\right]=0.
\]
Putting the above two equations together, we get the desired result.
\end{proof}

\end{document}